\pdfoutput=1
\documentclass[11pt,a4paper]{article}
\usepackage[utf8]{inputenc}
\usepackage{a4wide}
\usepackage{amsmath,amssymb,amsfonts,amsthm,mathtools,yhmath}
\usepackage[mathcal]{euscript}
\usepackage{mathrsfs}
\usepackage{xifthen}
\usepackage{enumitem}
\usepackage[usenames,dvipsnames]{xcolor}
\usepackage[unicode,bookmarks,colorlinks,breaklinks,pagebackref]{hyperref}
\hypersetup{linkcolor=OliveGreen,citecolor=OliveGreen,filecolor=black,urlcolor=Magenta}
\usepackage[capitalize,nameinlink,noabbrev]{cleveref}
\usepackage{cite}
\newcommand{\lemref}[1]{\hyperref[#1]{Lemma~\ref*{#1}}}
\newcommand{\propref}[1]{\hyperref[#1]{Proposition~\ref*{#1}}}
\newcommand{\corref}[1]{\hyperref[#1]{Corollary~\ref*{#1}}}
\newcommand{\remref}[1]{\hyperref[#1]{Remark~\ref*{#1}}}

\newtheorem{theorem}{Theorem}[section]
\newtheorem{lemma}[theorem]{Lemma}
\newtheorem{proposition}[theorem]{Proposition}
\newtheorem{corollary}[theorem]{Corollary}
\newtheorem{remark}[theorem]{Remark}
\newtheorem{assume}{Assumption}
\newtheorem{example}{Example}
\numberwithin{equation}{section}
\allowdisplaybreaks

\newcommand{\ZZ}{\mathbb{Z}}
\newcommand{\RR}{\mathbb{R}}
\newcommand{\CC}{\mathbb{C}}
\newcommand{\Rm}{\RR}
\newcommand{\Cm}{\CC}
\newcommand{\eps}{\varepsilon}
\newcommand{\per}{\mathrm{per}}
\newcommand{\rP}{{\rm P}}
\newcommand{\email}[1]{\href{mailto:#1}{\texttt{#1}}}
\newcommand{\abs}[1]{\left\lvert #1\right\rvert}
\newcommand{\norm}[1]{\left\lVert #1\right\rVert}
\newcommand{\cv}[1][]{%
  \ifthenelse{\isempty{#1}}{\xrightarrow[\hphantom{~2~}]{}}{\xrightarrow[\hphantom{~2~}]{#1}}}
\newcommand{\wcv}[1][]{%
  \ifthenelse{\isempty{#1}}{\xrightharpoonup[\hphantom{~2~}]{}}{\xrightharpoonup[\hphantom{~2~}]{#1}}}
\newcommand{\calB}{\mathcal{B}}
\newcommand{\calD}{\mathscr{D}}
\newcommand{\calF}{\mathcal{F}}
\newcommand{\opH}{\mathcal{H}}
\newcommand{\opK}{\mathcal{K}}
\newcommand{\opM}{\mathcal{M}}
\newcommand{\ftf}{\mathscr{F}}
\DeclareMathOperator{\spec}{spec}
\DeclareMathOperator{\Index}{Index}
\DeclareMathOperator{\di}{d\!}
\DeclareMathOperator{\Div}{div}
\DeclareMathOperator{\sign}{sign}
\DeclareMathOperator{\dom}{dom}

\begin{document}

\title{Topological Anderson Insulators by homogenization theory}

\author{
  Guillaume Bal\footnote{Departments of Statistics and
    Mathematics and Committee on Computational and Applied
    Mathematics, University of Chicago, Chicago, IL 60637
    (\email{guillaumebal@uchicago.edu}). }
  \and
  Thuyen Dang\footnote{Department of Statistics and Committee on
    Computational and Applied Mathematics, University of Chicago,
    Chicago, IL 60637, USA. Present address: Department of Mathematics
    and Statistics, Purdue University Northwest, Hammond, IN 46323
    (\email{ttdang@pnw.edu}).}  }

\date{}

\maketitle

\begin{abstract}
  A central property of (Chern) topological insulators is the presence of robust asymmetric transport along interfaces separating two-dimensional insulating materials in different topological phases. A Topological Anderson Insulator is an insulator whose topological phase is induced by spatial fluctuations.

This paper proposes a mathematical model of perturbed Dirac equations and shows that for sufficiently large and highly oscillatory perturbations, the system is in a different topological phase than the unperturbed model. In particular, a robust asymmetric transport  indeed appears at an interface separating perturbed and unperturbed phases.

The theoretical results are based on careful estimates of resolvent operators in the homogenization theory of Dirac equations and on the characterization of topological phases by the index of an appropriate Fredholm operator.

\end{abstract}

\\[5mm]
\noindent{\bf AMS:} 35B20, 35B27, 35B35, 35B40, 47A53.

\section{Introduction}
\label{sec:intro}

Topological Anderson Insulators are insulators whose topological phase is induced by randomness \cite{liTopologicalAndersonInsulator2009}. A theoretical explanation based on effective medium theory for such phenomena was proposed in \cite{grothTheoryTopologicalAnderson2009}. Their starting point is a randomly perturbed system of Dirac equations.

This paper considers a similar system of Dirac equations, given explicitly by \eqref{eq:1} below, that includes periodic perturbations of the form $\varepsilon^{-1}V(\eps^{-1} x)$ with $V$ a matrix-valued periodic potential and $x=(x_1,x_2)\in\Rm^2$ spatial coordinates. Using homogenization theory, we show that in the limit $0<\eps\to0$, the system converges to a homogenized Dirac equation in a different topology from the unperturbed system formally obtained as $\eps\to\infty$. The topological phase of the homogenized operator is determined by an effective mass, which combines the homogenized mass term and a constant vector potential generated by correlations between the components of the fluctuation.

Arguably the most unexpected feature of topological insulators is the asymmetric transport observed at interfaces separating insulators in different topological phases. An interface in our system is modeled by a domain wall $\rho(x)=\rho(x_2)$ and perturbations of the form $\rho(x)\varepsilon^{-1}V(\eps^{-1} x)$.  This allows us to define a transition from the unperturbed system where $\rho(x_2)=0$ when $x_2\geq1$ to a perturbed system where $\rho(x_2)=1$ when $x_2\leq-1$. We will show that the interface Hamiltonian with such a domain wall indeed carries robust asymmetric transport for $0<\eps<\eps_0$ sufficiently small while no such robust asymmetric transport exists for $\eps>\eps_1>0$ sufficiently large. Since the interface index cannot change along a family of operators whose bulks remain gapped, the bulk operator describing the perturbed phase necessarily closes its spectral gap at some intermediate value of $\eps$ (\corref{cor:gapclose}).

\medskip

The invariants mentioned above are defined for large classes of topological insulators. See, e.g., \cite{ASS90,bellissard1994noncommutative,bernevig2013topological,bourne2017k,chiu2016classification,delplace,kitaev2009periodic,moessner2021topological,prodan2016bulk,thiang2016k,Volovik,WI} for general references from the mathematical and physical literature on the broad and actively studied field of topological insulators. The robust asymmetric transport observed at one-dimensional interfaces separating distinct two-dimensional insulators has also been analyzed in many contexts.  For many discrete and continuous Hamiltonians, a quantized interface current observable as well as a number of Fredholm operators with non-trivial indices may be introduced to compute topological invariants and relate them to the observed quantized, robust-to-perturbations, asymmetric transport along the interface; see, e.g. \cite{Elbau,prodan2016bulk,SB-2000} for discrete Hamiltonians and \cite{balContinuousBulkInterface2019,bal2022topological,bal2023topological,Drouot:19,Drouot:19b,drouot2020microlocal,Hatsugai,QB22} for partial- or pseudo- differential Hamiltonians.

In this paper, the systems of interest are modeled by Dirac equations acting on spinors (vector-valued functions) of the Euclidean plane $\Rm^2$. The topological invariants we will be using are Chern invariants for the bulk phases. For regularized Dirac operators, their theory is presented in \cite{balContinuousBulkInterface2019}. The corresponding invariants for the interface Hamiltonians modeling a transition between different bulk phases are analyzed in \cite{bal2022topological,bal2023topological,QB22}.

\medskip

The rest of the paper is structured as follows. Section \ref{sec:formulation} formulates the system of Dirac equations with periodic perturbations, defines the notions of bulk and interface invariants, and presents our main results. In particular, we first derive a homogenization result for the Dirac equation with periodic potential. The limiting homogenized operator is then shown to be in a different bulk topological phase than the unperturbed operator. Estimates on the resolvent operator in the homogenized limit then allow us to state our main result, namely that for $\eps$ sufficiently small, the perturbed Dirac operator with domain wall carries non-trivial edge transport.  The proofs of the main results are then presented in detail in section \ref{sec:main-results}.

\section{Formulation and main results}
\label{sec:formulation}
This section introduces our model and presents our main results. All proofs are postponed to section \ref{sec:main-results}.
\subsection{Dirac model with periodic fluctuations}
\label{sec:notation}
Throughout the paper, we use the following notation.
\begin{itemize}
\item All the functions are assumed to be complex-valued, unless
  stated otherwise.
\item $Y\coloneqq [0,1]^2$ -- the reference unit cell.
\item $L^p_{\per}(Y)$ for some $p \in [1,\infty]$ -- the space of
  $Y$-periodic functions that are $L^p$-integrable.
\item For $f = (f_1, f_2)^{\top}$ and $g = (g_1,g_2)^{\top}$ in
  $L^2(\RR^2;\CC^2)$, we define the inner product
\begin{align*}
  \left\langle f,g \right\rangle
  = \left\langle
  \begin{pmatrix}
    f_1\\f_2
  \end{pmatrix},
  \begin{pmatrix}
    g_1\\g_2
  \end{pmatrix}
  \right\rangle
  \coloneqq \int_{\RR^2} f_1 \overline{g_1} + f_2 \overline{g_2} \di x,
\end{align*}
and its corresponding norm $\norm{f} \coloneqq \left\langle f,f \right\rangle^{\frac{1}{2}}.$
\item $\ftf$ -- the Fourier transform, with
  $\xi \coloneqq (\xi_1, \xi_2)$ to be the dual variable to
  $x = (x_1,x_2)$, and $\left\langle \xi \right\rangle \coloneqq
  \left( 1 + \abs{\xi}^2 \right)^{\frac{1}{2}}$ to be the Japanese
  bracket.
\item The Einstein summation convention is used whenever applicable;
  $\delta_{ij}$ is the Kronecker delta, and $\epsilon_{ijk}$ is the
  permutation symbol.
\end{itemize}

The Dirac operator is then defined as follows.
Let $\varepsilon > 0$. We consider the following Hamiltonian on
$L^2(\RR^2) \otimes \CC^2 \cong L^2(\RR^2;\CC^2)$:
\begin{align}
\label{eq:1o}
  \opH^{\varepsilon} = \left( D + (U_1^{\varepsilon}, U_2^{\varepsilon})
  \right) \cdot \sigma + (m + \beta \Delta + U_3^{\varepsilon})
  \sigma_3 + U_0^{\varepsilon} \sigma_0,
\end{align}
where $m,\beta$ are real numbers with $|m|, |\beta|$ in $(0,\infty)$,  $\left\{ \sigma_i \right\}$ are Pauli matrices
\begin{align*}
  \sigma_1
  &\coloneqq
    \begin{pmatrix}
      0 &1\\
      1 &0
    \end{pmatrix}, \quad
    \sigma_2
  \coloneqq
    \begin{pmatrix}
      0 &-i\\
      i &0
    \end{pmatrix}, \quad
    \sigma_3
    \coloneqq
    \begin{pmatrix}
      1 &0\\
      0 &-1
    \end{pmatrix}, \quad
    \sigma_0
    \coloneqq \text{Id},\\
  x
  &\coloneqq (x_1, x_2) \in \RR^2,~
  D
  \coloneqq  \frac{1}{i} \left( \partial_{x_1}, \partial_{x_2}
      \right),~
  \sigma
  \coloneqq \left( \sigma_1, \sigma_2 \right),~\Delta \coloneqq
    \partial_{x_1}^2 + \partial_{x_2}^2,
\end{align*}
and
\begin{align}
\label{eq:4}
  U^{\varepsilon}_j (x) \coloneqq \frac{1}{\varepsilon} \rho (x)
  V_j^{\varepsilon}(x), \quad V_j^{\varepsilon}(x) \coloneqq V_j \left( \frac{x}{\varepsilon} \right),
\end{align}
with $\rho \in  W^{2,\infty}(\RR^2;\RR)$ and $V_j\in
C^{0,\alpha}_{\per}(Y;\RR)$ with $\int_Y V_j \di y = 0$, for some $ \alpha \in (0,1)$ and $j \in
\left\{ 0,1,2,3 \right\}.$

The form of $U^{\varepsilon}$ is motivated by the work on homogenization with
rapidly oscillating potentials in \cite[Chap. 1,
Sect. 12]{bensoussanAsymptoticAnalysisPeriodic2011}.  Homogenization problems with similar potentials were studied in
\cite{bensoussanAsymptoticAnalysisPeriodic2011,zhangHomogenizationSchrodingerEquation2014,balLimitingModelsEquations2015,balHomogenizationLargeSpatial2010,cancesSecondorderHomogenizationPeriodic2023,goudeyLinearEllipticHomogenization2022a,allaireHomogenizationSchrodingerEquation2005,ducheneScatteringHomogenizationInterface2011}. Operator-norm error estimates in periodic homogenization go back to \cite{BS04,ZP05}. The norm resolvent convergence of fine-scale divergence operators is studied in \cite{ZP16,CKS23}, and that of Schr\"odinger operators with oscillating potentials in \cite{cancesSecondorderHomogenizationPeriodic2023}. The scaling of the potential of order $\eps^{-1}$ ensures that the limiting homogenized limit is an order $O(1)$ modification of the unperturbed operator.

The domain wall separating the perturbed and unperturbed phases is modeled by the function $\rho$ in \eqref{eq:4}.

\begin{remark}
Except for the perturbations $U_1^{\varepsilon}, U_2^{\varepsilon}, U_3^{\varepsilon}$ and the
specific form of the perturbation $U_0^\eps$, the Dirac model is essentially the same as the one used in \cite{grothTheoryTopologicalAnderson2009}.

Varying the mass term $U^\eps_3$ may be challenging practically and is included for mathematical completeness. We will show that variations of the magnetic potential $(U_1^\eps,U_2^\eps)$ with $U_0^\eps=U_3^\eps=0$, or variations of the electric potential $U_0^\eps$ with $U_1^\eps=U_2^\eps=U_3^\eps=0$, are sufficient to generate asymmetric edge transport.
\end{remark}

\subsection{Homogenization theory}
\label{sec:homog-theory}

The following form of $\opH^{\varepsilon}$ is more suitable for
homogenization purposes. Observe that
\begin{align*}
  U^{\varepsilon}
  \coloneqq U^{\varepsilon}_1 \sigma_1 + U^{\varepsilon}_2 \sigma_2
  + U^{\varepsilon}_3 \sigma_3 + U^{\varepsilon}_0 \sigma_0
  =
 \frac{1}{\varepsilon}\rho (x) \begin{pmatrix}
    V_0^{\varepsilon} + V_3^{\varepsilon} & V_1^{\varepsilon} - i
                                            V_2^{\varepsilon} \\
    V_1^{\varepsilon} + i V_2^{\varepsilon} & V_0^{\varepsilon} - V_3^{\varepsilon}
  \end{pmatrix},
\end{align*}
so if we let
\begin{align}
  \label{eq:60}
W \coloneqq \begin{pmatrix}
    V_0 + V_3 & V_1 - i V_2 \\
    V_1 + i V_2 & V_0 - V_3
  \end{pmatrix},%
\end{align}
then we can rewrite \eqref{eq:1o} as
\begin{align}
  \label{eq:1}
  \opH^{\varepsilon}
  = D \cdot \sigma + \left( m + \beta \Delta \right) \sigma_3 +
  \frac{1}{\varepsilon} \rho (x) W \left( \frac{x}{\varepsilon} \right).
\end{align}

The unperturbed (regularized) Dirac operator from $L^2(\RR^2;\CC^2)$ to $L^2(\RR^2;\CC^2)$ is given explicitly by
\begin{align*}
  \opH^{\infty}
  \coloneqq D \cdot \sigma + \left( m + \beta \Delta \right) \sigma_3
\end{align*}
together with its domain
\begin{align*}
  \dom (\opH^{\infty})
  \coloneqq \left\{ f \in L^2(\RR^2;\CC^2)\colon \opH^{\infty} f \in L^2(\RR^2;\CC^2) \right\}.
\end{align*}
We have the result:
\begin{lemma}
\label{lem:main-results-4}
$\dom(\opH^{\infty}) = H^2(\RR^2;\CC^2)$ and  $\opH^{\infty}\colon \dom (\opH^{\infty}) \subset L^2(\RR^2;\CC^2) \to L^2 (\RR^2;\CC^2)$ is self-adjoint.
\end{lemma}
For each $\varepsilon > 0$,
the multiplication operator $\opM_{U^{\varepsilon}}\coloneqq
U^{\varepsilon} \sigma_0$ is bounded and self-adjoint on
$L^2(\RR^2;\CC^2)$ by the assumptions on $U^{\varepsilon}$, therefore,
by \cite[Proposition 1.6]{schmudgenUnboundedSelfadjointOperators2012},
\begin{align*}
  (\opH^{\varepsilon})^{*}
  =(\opH^{\infty}+ \opM_{U^{\varepsilon}})^{*}
  = (\opH^{\infty})^{*} +
  (\opM_{U^{\varepsilon}})^{*}
  =\opH^{\infty} +
  \opM_{\overline{U^{\varepsilon}}}
  = \opH^{\infty} + \opM_{U^{\varepsilon}} = \opH^{\varepsilon},
\end{align*}
or $\opH^{\varepsilon}$ is self-adjoint with $\dom(\opH^{\varepsilon})
= \dom (\opH^{\infty})$. Consequently, the spectrum
\begin{align}
\label{eq:12}
\spec (\opH^{\varepsilon}) \subset \RR.
\end{align}

To eliminate $\varepsilon^{-1}$ in \eqref{eq:1}, we introduce the auxiliary problems
\begin{align}
\label{eq:5}
\Delta_y \phi_{kl} = W_{kl}, \qquad \phi_{kl} \in H_{\per}^1(Y;\CC)/\RR, ~ 1
  \le k, l \le 2.
\end{align}
Elliptic regularity theory \cite{gilbargEllipticPartialDifferential2001} implies $\phi_{kl} \in H_{\per}^2(Y;\CC)/\RR$.
Define
\begin{align}
\label{eq:6}
\Phi_{kl}(y) \coloneqq \nabla_y \phi_{kl} (y) \ \text{ and }\ \Phi_{kl}^{\varepsilon}(x) \coloneqq \Phi_{kl}
  \left( \frac{x}{\varepsilon} \right),\qquad 1
  \le k, l \le 2.
\end{align}
Then
\begin{align*}
  \Div  \Phi_{kl}^{\varepsilon}(x)
  = \Div \Phi_{kl} \left( \frac{x}{\varepsilon} \right) =
  \frac{1}{\varepsilon} \Div_y \left( \Phi_{kl}(y) \right) \Big\vert_{y =
  \frac{x}{\varepsilon}} = \frac{1}{\varepsilon} \Div_y \left(
  \nabla_y \phi_{kl} (y) \right)\Big\vert_{y =
  \frac{x}{\varepsilon}} =
  \frac{1}{\varepsilon} W_{kl} \left( \frac{x}{\varepsilon} \right).
\end{align*}
Therefore,
\begin{align}\label{eq:2}
  U^{\varepsilon} (x) = \frac{1}{\varepsilon}\rho(x) W_{kl} \left( \frac{x}{\varepsilon} \right)
  = \rho(x) \Div \Phi_{kl}^{\varepsilon}(x) =  \Div \left( \rho(x) \Phi_{kl}^{\varepsilon}(x) \right) - \nabla \rho(x) \Phi_{kl}^{\varepsilon}(x)
\end{align}
since
\begin{align*}
  \Div \left( \rho (x) \Phi_{kl}^{\varepsilon}(x) \right)
  &= \rho(x) \Div \Phi_{kl}^{\varepsilon} (x) + \nabla \rho(x) \cdot \Phi_{kl}^{\varepsilon}
  (x)= \frac{1}{\varepsilon} \rho(x) W_{kl} \left( \frac{x}{\varepsilon} \right) + \nabla \rho(x) \cdot \Phi_{kl}^{\varepsilon}(x).
\end{align*}
\begin{theorem}
  \label{thm:main-results-2}
  Let $\varepsilon > 0$. Let $\lambda > 0$, $z \in \CC$ with $\Re z \in [-\lambda,\lambda]$. There
  exist $\gamma = \gamma (\lambda, m, \beta, \norm{\rho}_{L^{\infty}},
  \norm{W}_{C^{0,\alpha}}) > 0$ and $C = C (\lambda,  m, \beta, \norm{\rho}_{L^{\infty}},
  \norm{W}_{C^{0,\alpha}}) > 0$, both independent of $\varepsilon$,
  such that whenever $\Im z \in [-\gamma, \gamma] \setminus \left\{ 0 \right\}$ the equation
\begin{align}
  \label{eq:7}
  \psi^{\varepsilon,z} \in \dom(\opH^{\infty}),\qquad
 (\opH^{\varepsilon} - z ) \psi^{\varepsilon,z}
  =  f,%
\end{align}
has a unique solution $\psi^{\varepsilon,z} \in \dom(\opH^{\infty})$ that
satisfies
\begin{align}
\label{eq:9}
  \norm{\psi^{\varepsilon,z}}_{H^1}
  \le C\abs{\Im z}^{-1} \norm{f}.
\end{align}

For $1 \le k,l \le 2$, consider the cell problem
\begin{align}
\label{eq:25}
T_{kl} \in H^1_{\per}(Y; \CC), \quad \beta \Delta_{yy} T_{kl} + W_{kl} = 0.
\end{align}
Define
\begin{align*}
T(y) \coloneqq \begin{pmatrix}
    T_{11}(y) & T_{12}(y)\\
    T_{21}(y) & T_{22} (y)
  \end{pmatrix}, \qquad y \in Y.
\end{align*}
Then as $\varepsilon \to 0$, the solution $\psi^{\varepsilon,z}$ of \eqref{eq:7} weakly converges to
$\psi^{s,z}$  in $H^1(\RR^2;\CC^2)$ such that
\begin{align}
\label{eq:23}
(\opH^{0}-z) \psi^{s,z} = f,
\end{align}
where the homogenized operator $\opH^0$ is defined by
\begin{align}
\label{eq:24}
  \opH^{0}
  &\coloneqq
   D \cdot \sigma + \left(m + \beta \Delta\right)
    \sigma_3 + \beta \rho^2(x) \tau,\\
  \label{eq:30}
  \tau
  &\equiv
    \begin{pmatrix}
      \tau_{11}&\tau_{12}\\
      \tau_{21} & \tau_{22}
    \end{pmatrix} \coloneqq \frac{1}{\beta}\int_Y W(y) \sigma_3 T(y) \di y
    \nonumber \\
  &=
    \int_Y
    \begin{pmatrix}
      \abs{\nabla T_{11}}^2 - \nabla T_{21} \nabla T_{12} & \nabla T_{11}
    \nabla T_{12} -\nabla T_{12} \nabla T_{22}\\
      \nabla T_{11} \nabla T_{21} - \nabla T_{21} \nabla T_{22} &
   -\abs{\nabla T_{22}}^2 +  \nabla  T_{12} \nabla T_{21}
    \end{pmatrix}
    \di y.
\end{align}
\end{theorem}
\begin{remark}
\label{sec:limit-absorpt-princ-3}
The homogenization theorem above can be extended to:
\begin{enumerate}
\item\label{item:1} The case $m = m(x)$ is a smooth, compactly supported
  perturbation of a constant, by the Kato--Rellich criterion.
\item The case when $\beta \Delta \cdot$ is replaced by $\Div a \left(
    \frac{x}{\varepsilon} \right) \nabla \cdot $ with a bounded
  elliptic matrix $a$. In this case, \eqref{eq:25} is replaced by
  $\Div [a \left( y \right) \nabla T_{kl} (y)] + W_{kl} =0$, and we need two additional cell
  problems $\Div a \left( y \right) [\nabla \chi_i (y) + e_i]=0,$ where
$e_i$ form a canonical basis of $\RR^2$.
\end{enumerate}
\end{remark}

\subsection{Resolvent estimates}
The above homogenization results, which describe the limiting homogenized operator, are not sufficiently strong to define topological invariants. We require stronger results on the convergence of resolvent operators that we now state. A first result concerns regularity of \eqref{eq:23}, which is necessary for the proof of norm resolvent convergence of $\opH^{\varepsilon}$.
\begin{lemma}
\label{lem:limit-absorpt-princ-1}
Let
$z \in [-\lambda,\lambda] + i ([-\gamma,\gamma] \setminus \left\{ 0
\right\})$ where $\lambda, \gamma$ are defined in
\cref{thm:main-results-2}. For each $f \in L^2(\RR^2;\CC^2)$, the
solution $\psi^{s,z}$ of \eqref{eq:23} satisfies
\begin{align}
\label{eq:47b}
  \norm{\psi^{s,z}}_{H^2}
  \le C \abs{\Im z}^{-1} \norm{f},
\end{align}
for some $C = C (\lambda,  m, \beta, \norm{\rho}_{L^{\infty}},
  \norm{W}_{C^{0,\alpha}})> 0$.
\end{lemma}
Elliptic regularity theory
\cite{gilbargEllipticPartialDifferential2001} also implies an $H^2$ bound
for $\psi^{s,z}$, but our
\lemref{lem:limit-absorpt-princ-1} specifies the dependence on $z$ in
the right hand side of the estimate. This is crucial for the following result:
\begin{theorem}
  \label{thm:limit-absorpt-princ-4}
Let $\lambda > 0$. There exist $C > 0$ and $\varepsilon_0' > 0$,
depending on $\lambda$, $m$, $\beta$, $\norm{\rho}_{W^{2,\infty}}$ and $\norm{W}_{C^{0,\alpha}}$, such that for any $f
\in L^2(\RR^2;\CC^2)$, $\varepsilon \in (0,\varepsilon_0'),$ and $z \in
\CC$ with $\Re z \in [-\lambda,\lambda],~\Im z \ne 0$, we have
\begin{align}
\label{eq:29}
\norm{(\opH^{\varepsilon}  - z )^{-1} - (\opH^{0} - z
  )^{-1}}_{L^2 \to L^2} \le C \left( 1+ \abs{\Im z}^{-2} \right) \varepsilon.
\end{align}
\end{theorem}

\subsection{Topological classification}
\label{sec:topol-interf-curr}

Let $\opH^{\infty}, \opH^{0}$ and $\opH^{\varepsilon}$ be the unperturbed, homogenized, and perturbed operators defined above. To compute the topological invariants, it is more convenient to rewrite $\tau$ in \eqref{eq:30} as
\begin{align*}
\tau = \tau_1 \sigma_1 + \tau_2 \sigma_2 + \tau_3 \sigma_3 + \tau_0
  \sigma_0,
\end{align*}
with
\begin{align}
\label{eq:3}
  \begin{split}
    \tau_0
    &\coloneqq \frac{1}{2} \int_Y \left( \abs{\nabla T_{11}}^2 -
      \abs{\nabla T_{22}}^2 \right) \di y,\\
    \tau_1
    &\coloneqq \frac{1}{2} \int_Y \left( \nabla T_{11} \nabla T_{12} +
      \nabla T_{11} \nabla T_{21} - \nabla T_{12} \nabla T_{22} -
      \nabla T_{21} \nabla T_{22}\right) \di y,\\
    \tau_2
    &\coloneqq \frac{1}{2i} \int_Y \left(- \nabla T_{11} \nabla T_{12} +
      \nabla T_{11} \nabla T_{21} + \nabla T_{12} \nabla T_{22} -
      \nabla T_{21} \nabla T_{22}\right) \di y,\\
    \tau_3
    &\coloneqq \frac{1}{2} \int_Y \left( \abs{\nabla T_{11}}^2 +
      \abs{\nabla T_{22}}^2 - 2 \nabla T_{12} \nabla T_{21} \right)
      \di y.
  \end{split}
\end{align}
It follows from \eqref{eq:24} that
\begin{align}
\label{eq:52}
  \opH^0
  = (D + \beta \rho^2(x)(\tau_1, \tau_2))\cdot \sigma + (m + \beta \Delta + \beta \rho^2(x)\tau_3)
  \sigma_3 + \beta \rho^2(x)\tau_0 \sigma_0.
\end{align}
In the sequel, we suppose
\begin{assume}[On domain wall]
  \label{sec:topol-class-3}
 There are $a < 0 < b$ so that
\begin{align}
\label{eq:16}
  \rho \in W^{2,\infty}(\RR^2;\RR),~
  \rho(x) = \rho(x_2) =
  \begin{cases}
    0, \quad x_2 \ge b,\\
    1, \quad x_2 \le a.
  \end{cases}
\end{align}
\end{assume}

We introduce the necessary ingredients to define invariants that
characterize the topological phases of Dirac operators following
\cite{balContinuousBulkInterface2019,bal2022topological,bal2023topological,QB22}. We define
\begin{equation}\label{eq:mpm}
m_+=m,\qquad m_-=m+\beta \tau_3,\qquad \tau_\perp\coloneqq(\tau_1,\tau_2),\qquad m_{\rm eff}\coloneqq m_--\beta^3\abs{\tau_\perp}^2 .
\end{equation}
The coefficients $\tau_j$ are real, since $\tau$ is Hermitian (see below). We will see that $m_{\rm eff}$, and not $m_-$, determines the topological phase of the homogenized operator.
We denote by $\opH^0_B$ the homogenized bulk (constant coefficient) operator
\[
  \opH^0_B = (D + \beta \tau_\perp)\cdot \sigma + (m_-+\beta
  \Delta)\sigma_3 + \beta \tau_0 \sigma_0.
\]

The operator $\opH^0_B$ thus represents the insulator on the lower
half-plane $x_2 \le a$, where the perturbations apply, i.e.,
$\rho (x_2) = 1$. There are no perturbations on the upper half-plane
$x_2 \ge b$, so the insulator on this part is still represented by
$\opH^{\infty}$.  Intuitively, asymmetric transport appears on the
domain wall $\RR\times [a,b]$ if the following conditions are
satisfied:
\begin{enumerate}
\item[\rm (i)]\label{item:6} Both insulators have a common spectral
  gap at the domain wall, more specifically,
  $\RR \setminus \spec (\opH^{\infty}) \cap \RR \setminus \spec
  (\opH_B^0)$ contains a (non-empty) open interval.
\item[\rm (ii)] $\opH^{\infty}$ and $\opH^0_{B}$ have different `topologies'. The latter are characterized by indices of Fredholm operators that will be introduced shortly.
\end{enumerate}

By definition, we have
\begin{align*}
\opH^0_B = \opH^{\infty} + \beta  \tau.
\end{align*}
The matrix $\tau$ is Hermitian by \eqref{eq:30}, \eqref{eq:25}, and \eqref{eq:60}. Let $p_A(t) \coloneqq \det(t - A)$ denote the characteristic polynomial of a given square matrix $A$.
Direct computations and the fact that $\tau_{12} = \overline{\tau_{21}}$ give
\begin{align*}
  p_{\widehat{\opH^{\infty}}} (s)
  &= s^2 - \left( m - \beta \abs{\xi}^2 \right)^2 - \abs{\xi}^2,\\
  p_{\widehat{\opH^0_{B}}}(s)
  &= s^2 - \beta (\tau_{11} + \tau_{22}) s - \left( \left( m-\beta
    \abs{\xi}^2 \right)^2 + \beta \left( m - \beta \abs{\xi}^2 \right)
    (\tau_{11}-\tau_{22}) - \beta^2 \tau_{11} \tau_{22}\right)\\
  &\qquad- \abs{\xi + \beta \tau_{12}}^2.
\end{align*}

Given $V_j,~0\le j \le 3$, one can compute $\tau$ by the formula
\eqref{eq:30} and obtain the eigenvalues of both $\opH^{\infty}$ and
$\opH^0_B$ as functions of the dual variable $\xi$. The spectrum of
$\opH^{\infty}$ and $\opH^0_{B}$ are the essential ranges of these
functions. By a common shift, we can assume whenever $\opH^{\infty}$
and $\opH^0_B$ have a common spectral gap, it contains 0.

\begin{assume}[Common spectral gap]\label{sec:topol-class-1}
  There exists $m_0 > 0$ such that
\begin{align}
\label{eq:32}
(-m_0, m_0) \ \ \subset \ \ \RR \setminus \spec (\opH^{\infty}) \ \cap\  \RR \setminus \spec
  (\opH_B^0).
\end{align}
\end{assume}
Equivalently, the symbol of $\opH^0_B$ is $\beta\tau_0+d(\xi)\cdot(\sigma_1,\sigma_2,\sigma_3)$ with
\begin{equation}\label{eq:dxi}
d(\xi)\coloneqq\bigl(\xi+\beta\tau_\perp,\ m_--\beta\abs{\xi}^2\bigr),
\end{equation}
whose eigenvalues are $\beta\tau_0\pm\abs{d(\xi)}$, while the symbol of $\opH^\infty$ has eigenvalues $\pm\abs{(\xi,m-\beta\abs{\xi}^2)}$. \Cref{sec:topol-class-1} is therefore equivalent to
\[
m_0\le\inf_\xi\abs{(\xi,m-\beta\abs{\xi}^2)}\qquad\mbox{and}\qquad \abs{\beta\tau_0}+m_0\le\inf_\xi\abs{d(\xi)}.
\]
Since $d_1$ and $d_2$ vanish only at $\xi=-\beta\tau_\perp$, where $d_3=m_{\rm eff}$, \cref{sec:topol-class-1} implies $m\ne0$ and $m_{\rm eff}\ne0$.
\begin{example}[Non vanishing electric potential]
\label{sec:common-spectral-gap}
When $V_1 = V_2 = V_3 = 0$ and $V_0 \ne 0$, then \eqref{eq:25} implies
$T_{11} = T_{22}$ and $T_{12}= T_{21} = 0$, so by \eqref{eq:3} we have
$\tau_0 = \tau_1 = \tau_2 = 0$,
$\tau_3 = \int_Y \abs{\nabla T_{11}}^2 \di y = \int_Y \abs{\nabla
  T_{22}}^2 \di y$. Hence $\tau_\perp=0$ and $m_{\rm eff}=m_-$, and \cref{sec:topol-class-1} is satisfied whenever $m_-\ne0$. We show below that
asymmetric transport \emph{appears} when $m \beta < 0$ is sufficiently large but \emph{cannot appear} when $m \beta > 0$.
\end{example}

\begin{example}[Non vanishing magnetic potential]
\label{sec:common-spectral-gap-1}
When $V_1 = V_3 = V_0 = 0$ and $V_2 \ne 0$, then $T_{11} = T_{22} = 0$ and
$T_{12} = -T_{21}$. From \eqref{eq:3}, we have
$\tau_0 = \tau_1 = \tau_2 = 0$ and
$\tau_3 = -\int_Y \abs{\nabla T_{12}}^2 \di y$. In this case,
$\tau_\perp=0$, $m_{\rm eff}=m_-$, \cref{sec:topol-class-1} holds whenever $m_-\ne0$, and asymmetric transport
\emph{appears} when $m \beta > 0$ is sufficiently large while it  \emph{cannot appear} when $m \beta < 0$.
\end{example}

We now express $\tau$ in terms of the components $V_j$ of the fluctuation, to quantify how they affect the homogenized operator $\opH^0$. Consider the auxiliary problems
\begin{align}
\label{eq:57}
t_j \in H_{\per}^1(Y;\RR), \qquad \beta \Delta_{yy} t_j + V_j = 0,
  \qquad j =0,1,2,3.
\end{align}
By \eqref{eq:60}, \eqref{eq:25} and uniqueness, we conclude
\begin{align}
\label{eq:61}
T_{11}=t_0+t_3,\qquad T_{22}=t_0-t_3,\qquad T_{12} = t_1 - it_2, \qquad T_{21} = t_1 + i t_2.
\end{align}
Inserting \eqref{eq:61} in \eqref{eq:3}, and using that the $t_j$ are real-valued, we obtain
\begin{align}
\label{eq:62}
\begin{split}
&\tau_0=2\int_Y\nabla t_0\cdot\nabla t_3\di y,\qquad
\tau_\perp=2\int_Y\nabla t_3\cdot\nabla (t_1,t_2)\di y,\\
&\tau_3=\int_Y\bigl(\abs{\nabla t_0}^2+\abs{\nabla t_3}^2-\abs{\nabla t_1}^2-\abs{\nabla t_2}^2\bigr)\di y .
\end{split}
\end{align}
In particular, the vector potential $\tau_\perp$ arises only from correlations of $V_3$ with $(V_1,V_2)$, and
\[
m_{\rm eff}=m+\beta\tau_3-\beta^3\abs{\tau_\perp}^2 .
\]

\begin{remark}
  We will justify the following statements. When $\tau_\perp=0$, as in the two examples above, $m_{\rm eff}=m_-$, and asymmetric transport occurs when $m\beta\tau_3 < 0$ is sufficiently
  large. More precisely, when $m \beta < 0$ (resp. $m \beta > 0$) and
  $\abs{\tau_3}$ is sufficiently large, then perturbations in front of
  $\sigma_3$ or $\sigma_0$ (resp. $\sigma_1$ or $\sigma_2$) change the
  `topology' of the regularized Dirac operator resulting in
  TAI/asymmetric transport. Perturbations in front of
  $(\sigma_1,\sigma_2)$ and $(\sigma_3,\sigma_0)$ have opposite
  effects on the change of topology. When $\tau_\perp\ne0$, the vector potential shifts the effective mass by $-\beta^3\abs{\tau_\perp}^2$, and it may reverse these conclusions (\remref{rem:vecpot}).
\end{remark}

Those statements will be formalized below, and our main result will be
stated in \cref{thm:stab}.

\paragraph{Limiting bulk invariants.}
The invariant characterizing bulk phases \cite{balContinuousBulkInterface2019,prodan2016bulk} is given by
\[
  F[H] = \Pi(H<0) \frac{x_1+ix_2}{|x_1+ix_2|} \Pi(H<0) + I-\Pi(H<0)
\]
where $\Pi(H<0)$ is the orthogonal projector onto the negative spectrum of  a self adjoint operator $H$ gapped at $0$ (i.e., $0\not\in \sigma(H)$).  Following \cite{balContinuousBulkInterface2019}, we have
\begin{proposition}[Bulk invariants] \label{prop:BI} The operators of the bulk Hamiltonians $F[\opH^\infty]$ and $F[\opH^0_B]$ are Fredholm on $L^2(\Rm^2;\Cm^2)$. Moreover,
\begin{equation}\label{eq:bulkindex}
   {\rm Index}\ F[\opH^\infty] = \frac12(\sign{m_+}+\sign{\beta}),\quad
   {\rm Index}\ F[\opH^0_B] = \frac12(\sign{m_{\rm eff}}+\sign{\beta}).
\end{equation}
\end{proposition}
We recall the index of a Fredholm operator $F$ is dim Ker $F\,-\,$dim Ker $F^*$. The formula for $\opH^\infty$ is that of \cite{balContinuousBulkInterface2019}, and the formula for $\opH^0_B$ is reduced to it in section~\ref{sec:main-results}.
This result shows that the presence of fluctuations changes the
topological bulk properties for the homogenized operator compared to
the unperturbed operator provided $m_+$ and $m_{\rm eff}$  in \eqref{eq:mpm}
have different signs. When $\tau_\perp=0$, this means that $m\beta\tau_3 < 0$  is sufficiently large. We have thus generated a topological insulator whose phase is directly related to the presence of the fluctuations $U^\eps$, at least asymptotically in the limit $\eps\to0$.
\begin{remark}
  The result stated in the above proposition holds generally for $m\not=0$ and $\beta\not=0$. As soon as $\sign{(m_+m_{\rm eff})}=-1$, we observe a change of bulk topology  ${\rm Index}\ F[\opH^\infty]\not={\rm Index}\ F[\opH^0_B]$. When $\tau_\perp=0$, this occurs for $m\beta<0$ when $\tau_3>0$ and for $m\beta>0$ when $\tau_3<0$ for $|\tau_3|$ sufficiently large.
\end{remark}

\begin{remark}[Vector potential]\label{rem:vecpot}
Because of the term $\beta\Delta$, the constant vector potential $\beta\tau_\perp$ cannot be removed from $\opH^0_B$ by a gauge transformation: conjugation by $e^{i\beta\tau_\perp\cdot x}$ removes it from the first-order part but adds a first-order term to the mass (see the proof of \propref{prop:BI} in section~\ref{sec:main-results}). It modifies the index through $m_{\rm eff}$, and it may reverse the conclusions above. Take $m=\beta=1$, $V_0=V_2=0$ and $V_1=V_3=\lambda v$, with $v\ne0$ real, periodic, and of mean zero, and let $-\Delta\phi=v$ and $E=\int_Y\abs{\nabla\phi}^2\di y$. Then $t_1=t_3=\lambda\phi$ and $t_0=t_2=0$, so that $\tau_0=\tau_2=\tau_3=0$ and $\tau_1=2\lambda^2E$ by \eqref{eq:62}. Hence $m_-=m_+=1$, while $m_{\rm eff}=1-4\lambda^4E^2$. For $4\lambda^4E^2>1$, \cref{sec:topol-class-1} holds and the interface index of $\opH^0$ is $-1$ by \lemref{lem:index} below.
\end{remark}

\paragraph{Limiting interface invariants.} We now consider the
practically more relevant case of an interface invariant
characterizing quantized asymmetric transport along the axis $x_2=0$.
To define a transition between insulators in different phases, we
assume that $\rho$ satisfies \cref{sec:topol-class-3}.

Recall the constant $m_0$ of \cref{sec:topol-class-1}. The constraint $m_0>0$ implies that the bulk Hamiltonians display a spectral gap in the interval $(-m_0,m_0)$. Therefore, excitations in this energy range are constrained to remain in the vicinity of $x_2=0$ while excitations outside of that range may propagate into the bulk. An interface invariant necessarily focuses on the former excitations. We thus define a non-decreasing smooth function $\varphi(h)$ equal to $1$ for $h\geq m_0$ and equal to $0$ for $h\leq -m_0$. In other words, $\varphi'(h)\geq0$ is supported in $(-m_0,m_0)$, integrates to $1$, and may be interpreted as a density of states that will remain localized in the vicinity of the interface $x_2=0$.

More precisely, for $\opH$ an unbounded  self-adjoint operator on $L^2(\RR^2;\CC^2)$, then $\varphi'(\opH)$ is a bounded operator modeling a density of states supported in the energy range $(-m_0,m_0)$ and  $U(\opH)=e^{2\pi i \varphi(\opH)}$ is a unitary operator. The function $S(h)=U(h)-1$ is thus also smooth and compactly supported in $(-m_0,m_0)$.

In \cite{bal2022topological,bal2023topological,QB22} (see also \cite{Elbau,prodan2016bulk,SB-2000}), the following interface invariant is defined
\[
  2\pi \sigma_I[\opH ] = {\rm Tr}\, 2\pi i[\opH,P] \varphi'(\opH).
\]
Here $P(x)=P(x_1)$ is a smooth function equal to $0$ for $x_1\leq x_0-\delta$ and equal to $1$ for $x_1\geq x_0+\delta$ for some $x_0\in\Rm$ and $\delta>0$. The operator $i[\opH,P]$ may be interpreted as a current operator (modeling transport from the left of $x_1=x_0$ to the right of $x_1=x_0$ per unit time) so that the above trace may indeed be interpreted as the expectation of a current operator against the density of states $\varphi'(\opH)$. Any non-vanishing value indicates asymmetric transport.

Since $\opH$ is elliptic, it is shown in \cite{bal2022topological,QB22} that $2\pi i[\opH,P] \varphi'(\opH)$ is indeed a trace-class operator for the  interface Hamiltonians $\opH=\opH^\infty$ and $\opH=\opH^0$.  In fact, since the unperturbed Hamiltonian $\opH^\infty$ is gapped in $(-m_0,m_0)$, we obviously have that $2\pi \sigma_I[\opH^\infty ] =0$. More precisely, we have \cite{bal2022topological,QB22}
\begin{proposition}[Interface invariants] Under the hypotheses of \propref{prop:BI}, we have
\begin{equation}\label{eq:interfaceindex}
   2\pi \sigma_I[\opH^\infty] = 0\quad\mbox{ and } \quad
   2\pi \sigma_I[\opH^0] = \frac12(\sign{m_{\rm eff}}-\sign{m_+}).
\end{equation}
\end{proposition}
The references \cite{bal2022topological,QB22} also show that the above interface current is stable against local perturbations of the Hamiltonian. This justifies the robustness of the asymmetric current against perturbations.

\paragraph{Stability of the interface invariant.}
So far, the invariants are only defined in the limits $\eps=0$ and $\eps=\infty$. We now show that invariants remain defined and stable for $\eps$ sufficient close to $0$ and $\eps$ sufficiently large.

Let $\rP(x)={\rm H}(x_1-x_0)$ be a Heaviside function equal to $1$ for $x_1\geq x_0$ and $0$ for $x_1<x_0$. We can then construct the family of bounded operators
\begin{equation}\label{eq:Teps}
   T^{\varepsilon} :=  T[\opH^\eps] = \rP(x) U(\opH^{\varepsilon}) \rP(x) + (I-\rP(x)),\qquad 0 \leq \varepsilon \leq \infty.
\end{equation}
From \cite{bal2022topological,QB22}, we obtain that
\begin{lemma} \label{lem:index}
  At $\varepsilon=\infty$ and $\varepsilon=0$, the operators $T^\infty=T[\opH^\infty]$ and $T^0=T[\opH^0]$ are Fredholm operators. Moreover,
  \[
  {\rm Index}\, T^\infty =2\pi \sigma_I[\opH^\infty] = 0 ,\qquad {\rm
    Index}\, T^0 =2\pi \sigma_I[\opH^0] =
  \frac12(\sign{m_{\rm eff}}-\sign{m_+}).%
  \]
\end{lemma}
As mentioned above, the unperturbed operator $\opH^\infty$ admits a spectral gap in $(-m_0,m_0)$ so that $U(\opH^{\infty})=I$ and $T^\infty=I$ is clearly a Fredholm operator with trivial index.  The result for $T^0$ is more interesting; when $\tau_\perp=0$, it was also obtained in \cite{balContinuousBulkInterface2019}. The proofs of \propref{prop:BI}, of the formula for $2\pi\sigma_I[\opH^0]$, and of \lemref{lem:index} for $T^0$ are given in section~\ref{sec:main-results}.

We can now state the main result of this paper:
\begin{theorem}\label{thm:stab}
    For $\varepsilon_1$ sufficiently large and $\varepsilon_0$ sufficiently small, then for all $0<\eps<\eps_0$ and for all $\eps_1<\eps$, $T^\eps$ in \eqref{eq:Teps}  is a Fredholm operator. Moreover, we have
\begin{equation}\label{eq:indexeps}
   {\rm Index}\, T^\varepsilon=0 ,\quad  \varepsilon_1 \leq \varepsilon<\infty, \qquad {\rm Index}\, T^\varepsilon =\frac12(\sign{m_{\rm eff}}-\sign{m_+}), \quad  0<\varepsilon \leq \varepsilon_0.
\end{equation}
\end{theorem}

\cref{thm:stab} shows that the topological invariants assigned to $\opH^{\varepsilon}$ for $\varepsilon$ sufficiently large and $\varepsilon$ sufficiently small are different when $\sign{(m_+m_{\rm eff})}=-1$. The map $\eps\mapsto\opH^\eps$ is not continuous in the norm-resolvent sense on $L^2(\RR^2;\CC^2)$, since $W(x/\eps)$ changes on all scales, so that local constancy of the index along $\eps$ is not available, and this does not imply by itself that $T^\eps$ fails to be Fredholm for some intermediate $\eps$. The correct consequence concerns the lower bulk operator
\begin{equation}\label{eq:HepsB}
\opH^\eps_B\coloneqq D\cdot\sigma+(m+\beta\Delta)\sigma_3+\frac1\eps W\Bigl(\frac x\eps\Bigr).
\end{equation}
\begin{corollary}[Gap closing]\label{cor:gapclose}
Let $0<\eps'<\eps''$ be such that $0\notin\spec\opH^{\eps'}_B\cup\spec\opH^{\eps''}_B$, and let the switch function $\varphi$ in \eqref{eq:Teps} have its transition in a common gap of $\opH^\infty$, $\opH^{\eps'}_B$ and $\opH^{\eps''}_B$ around $0$. If $T^{\eps'}$ and $T^{\eps''}$ are Fredholm with different indices, then $0\in\spec\opH^\eps_B$ for some $\eps\in(\eps',\eps'')$. In particular, under the assumptions of \cref{thm:stab}, if $\sign{(m_+m_{\rm eff})}=-1$, then, for $\eps_0$ small enough and $\eps_1$ large enough, $0\in\spec\opH^\eps_B$ for some $\eps\in(\eps_0,\eps_1)$.
\end{corollary}
The statement holds for any bounded $W$; a version for elliptic systems of arbitrary order is proved in \cite[Theorem~2.6]{bal26}. The proof is given in section~\ref{sec:main-results}.

\begin{remark}
    The latter result displays the main property we wanted to establish, namely that heterogeneous fluctuations in a half-space are sufficient to generate an asymmetric transport along the interface $x_2\sim0$.

    The terminology of `Anderson' topological insulator is somewhat misleading. As we just showed, periodic fluctuations (and not genuinely random fluctuations as in the mathematical derivations of Anderson localization \cite{FS83,GK01}) suffice to generate a topological change. Moreover, these fluctuations generate a non-trivial topology and hence a non-vanishing asymmetric transport, which may in fact be seen as an obstruction to Anderson localization \cite{topological}.
\end{remark}

\section{Proof of the homogenization and stability results}
\label{sec:main-results}
We now present a detailed proof of the main results stated in the preceding section.
\begin{proof}[Proof of Lemma \ref{lem:main-results-4}]
  \begin{enumerate}[wide]
  \item We diagonalize the constant coefficient operator $\opH^{\infty}$ by Fourier transform to obtain
\begin{align*}
\widehat{\opH^{\infty}}(\xi) = \xi \cdot \sigma + (m-\beta \abs{\xi}^2) \sigma_3,
\end{align*}
which is the symbol of $\opH^{\infty}$, i.e., $\opH^{\infty} = \ftf^{-1} \widehat{\opH^{\infty}} \ftf,$
see, e.g.,  \cite{balContinuousBulkInterface2019}. Direct computation shows that
\begin{align}
\label{eq:42}
  \begin{split}
    \widehat{\opH^{\infty}}(\xi)
    &= \xi_1
      \sigma_1
      + \xi_2
      \sigma_2
      + (m - \beta \abs{\xi}^2)
      \sigma_3
      = \begin{pmatrix}
          m - \beta \abs{\xi}^2 & \xi_1 - i \xi_2\\
          \xi_1 + i \xi_2 &-(m - \beta \abs{\xi}^2)
      \end{pmatrix}.
  \end{split}
\end{align}
Therefore,
\begin{align}
\label{eq:43}
 \widehat{\opH^{\infty}}^{*} = \widehat{\opH^{\infty}}
\end{align}
and
\begin{align}
\label{eq:44}
  \begin{split}
    \widehat{\opH^{\infty}}^2
    &= \left(  \abs{\xi}^2 + \left(m - \beta \abs{\xi}^2\right)^2
      \right) \sigma_0.
  \end{split}
\end{align}

\item\label{item:2} Clearly, $H^2(\RR^2;\CC^2) \subset \dom(\opH^{\infty}).$ So we only need
  to justify  $\dom (\opH^{\infty}) \subset H^2(\RR^2;\CC^2)$.
Let $f \in \dom (\opH^{\infty})$, then $\opH^{\infty} f \in L^2(\RR^2;\CC^2)$. We need $f \in H^2(\RR^2; \CC^2)$. To this end, we will show
$\left\langle \xi \right\rangle^2 \hat{f} \in L^2(\RR^2;\CC^2).$

By the Plancherel theorem,
$\norm{\widehat{\opH^{\infty}} \hat{f}} = \norm{\opH^{\infty} f}$, so
$\widehat{\opH^{\infty}} \hat{f}$ is in $L^2(\RR^2;\CC^2)$. Together
with \eqref{eq:43} and \eqref{eq:44}, we get
\begin{align*}
 \infty >  \norm{\widehat{\opH^{\infty}}\hat{f}}^2
  &= \left\langle \widehat{\opH^{\infty}} \hat{f}, \widehat{\opH^{\infty}} \hat{f}\right\rangle
    = \left\langle \widehat{\opH^{\infty}}^2 \hat{f},\hat{f} \right\rangle\\
  &=  \int_{\RR^2} \left(  \abs{\xi}^2 + \left(m - \beta \abs{\xi}^2\right)^2
    \right) \abs{\hat{f}(\xi)}^2 \di \xi\\
  &\ge \int_{\RR^2} \abs{\xi}^2 \abs{\hat{f}(\xi)}^2 \di \xi,
\end{align*}
so
\begin{align*}
  \infty >  (2 + 2 \abs{m\beta})\norm{\widehat{\opH^{\infty}}\hat{f}}^2
  &= \left\langle \widehat{\opH^{\infty}}^2 \hat{f},\hat{f}
    \right\rangle + (1 + 2
    \abs{m\beta})\norm{\widehat{\opH^{\infty}}\hat{f}}^2 \\
  &=  \int_{\RR^2} \left(  \abs{\xi}^2 + \left(m - \beta \abs{\xi}^2\right)^2
    \right) \abs{\hat{f}(\xi)}^2 \di \xi + (1 + 2
    \abs{m\beta})\norm{\widehat{\opH^{\infty}}\hat{f}}^2\\
  &\ge \int_{\RR^2} \left( \abs{\xi}^2 + m^2 - 2m\beta \abs{\xi}^2 +
    \beta^2 \abs{\xi}^4 \right) \abs{\hat{f}(\xi)}^2\di \xi\\
  &\qquad+
    \int_{\RR^2} (1+2 \abs{m\beta}) \abs{\xi}^2
    \abs{\hat{f}(\xi)}^2 \di \xi\\
  &\ge \int_{\RR^2} \left( 2\abs{\xi}^2 + m^2  +
    \beta^2 \abs{\xi}^4 \right) \abs{\hat{f}(\xi)}^2\di \xi\\
  &\ge \min \left\{ 1, m^2, \beta^2 \right\} \int_{\RR^2} \left( 1 +
    \abs{\xi}^2 \right)^2 \abs{\hat{f}(\xi)}^2 \di \xi.
\end{align*}
Therefore, $\left\langle \xi \right\rangle^2 \hat{f} \in L^2(\RR^2;\CC^2)$.
\item
  Note that $C_c^{\infty}(\RR^2;\CC^2) \subset \dom (\opH^{\infty})$
  and $C_c^{\infty}(\RR^2;\CC^2)$ is dense in $L^2(\RR^2;\CC^2)$, the
  operator $\opH^{\infty}$ is densely defined. From \eqref{eq:43},
  $\opH^{\infty}$ is also symmetric. Thus, to show $\opH^{\infty}$ is
  self-adjoint, we only need to show
  $(\opH^{\infty} \pm i) \dom(\opH^{\infty}) = L^2(\RR^2;\CC^2)$. The
  inclusion
  $(\opH^{\infty} \pm i) \dom(\opH^{\infty}) \subset L^2(\RR^2;\CC^2)$
  is obvious. We now prove $(\opH^{\infty} + i)
  \dom(\opH^{\infty}) = (\opH^{\infty} + i)H^2(\RR^2;\CC^2) \supset
  L^2(\RR^2;\CC^2)$.

  Let $g \in L^2(\RR^2;\CC^2)$ and $f = \ftf^{-1}\left(
    (\widehat{\opH^{\infty}} + i)^{-1} \hat{g} \right)$. The latter is
  well defined because $\pm i $ are not eigenvalues of
  $\widehat{\opH^{\infty}}(\xi)$ for any $\xi \in \RR^2$. Observe that
  \begin{align}
    \label{eq:64}
    \begin{split}
      \hat{f}
  &=  \left( \widehat{\opH^{\infty}} + i
    \right)^{-1} \hat{g}\\
  &= \left( \widehat{\opH^{\infty}} - i
    \right)\left( \widehat{\opH^{\infty}} - i
    \right)^{-1}\left( \widehat{\opH^{\infty}} + i
    \right)^{-1} \hat{g}\\
  &= \left( \widehat{\opH^{\infty}} -i \right) \left(
    (m- \beta \abs{\xi}^2)^2 + \abs{\xi}^2 + 1 \right)^{-1} \hat{g},
    \end{split}
\end{align}
Let $\omega (\xi) \coloneqq \frac{\abs{\left( 1 + \abs{\xi}^2\right) \left( \widehat{\opH^{\infty}} - i
  \right)}}{\left(m- \beta \abs{\xi}^2\right)^2 + \abs{\xi}^2 + 1}$
then $\omega \ge 0$ is continuous and there exist $\omega_0,
\omega_{\infty} \ge 0$ such that $\lim_{\abs{\xi}\to 0} \omega = \omega_0$ and
$\lim_{\abs{\xi} \to \infty} \omega = \omega_{\infty}. $ Thus there
exist $c_0, c_{\infty} > 0$ such that $\omega(\xi) \le \omega_0 +
\omega_{\infty} + 1$ whenever $\abs{\xi} \in [0,c_0] \cup [c_{\infty},
\infty)$. Continuity implies $\max_{\abs{\xi} \in [c_0, c_{\infty}]} \omega(\xi)$
exists. Thus \eqref{eq:64} implies
\begin{align*}
  \abs{\left( 1 + \abs{\xi}^2 \right)\hat{f}}
  \le \left( \max_{\abs{\xi}\in [c_0,
  c_{\infty}]} \omega(\xi) + \omega_0 + \omega_{\infty}
+ 1 \right)\abs{\hat{g}} \in L^2(\RR^2;\CC^2),
\end{align*}
so $\left( 1 + \abs{\xi}^2 \right)\hat{f} \in L^2(\RR^2;\CC^2),$ or $f
\in H^2(\RR^2;\CC^2).$
  Using
  $\opH^{\infty} = \ftf^{-1} \widehat{\opH^{\infty}} \ftf$ and the
  definition of $f$, we have $g = \left( \opH^{\infty} + i
  \right)f$. Therefore $L^2(\RR^2;\CC^2) \subset \left( \opH^{\infty}
    + i \right) H^2(\RR^2;\CC^2)$. The proof for $L^2(\RR^2;\CC^2) \subset \left( \opH^{\infty}
    - i \right) H^2(\RR^2;\CC^2)$ is similar.
\end{enumerate}
\end{proof}

\begin{proof}[Proof of \cref{thm:main-results-2}]
\begin{enumerate}[wide]
  \item Let $z \in \CC$ with $\Im z \ne 0$. From \eqref{eq:12}, $z$ belongs to the resolvent set of
    $\opH^{\varepsilon}$. Therefore, \eqref{eq:7} has a unique
    solution $\psi^{\varepsilon,z} \in \dom (\opH^{\infty}) = H^2(\RR^2;\CC^2)$. Moreover,
    applying the resolvent estimate for the self-adjoint operator
    $\opH^{\varepsilon}$ \cite[Theorem
    1.2.10]{daviesSpectralTheoryDifferential1995}, we obtain
\begin{align}
  \label{eq:45}
  \begin{split}
    \norm{\psi^{\varepsilon,z}}
  &
    \le \norm{(\opH^{\varepsilon}- z )^{-1}} \norm{f}
    \le \abs{\Im z}^{-1} \norm{f}.
  \end{split}
\end{align}
\item
 For $\eta \in \calD(\RR^2;\CC^2),$ we have from \eqref{eq:2} and
 integration by parts that
 \begin{align}
   \label{eq:63}
   \nonumber
   \left\langle U^{\varepsilon}\psi^{\varepsilon,z}, \sigma^3 \eta \right\rangle
&=\left\langle \frac{1}{\varepsilon} \rho (x) W \left(
  \frac{x}{\varepsilon} \right) \psi^{\varepsilon,z}, \sigma_3 \eta
  \right\rangle\\ \nonumber
  &= \left\langle \left[ \left( \Div (\rho(x)
    \Phi_{kl}^{\varepsilon}(x))- \nabla \rho(x) \cdot
    \Phi_{kl}^{\varepsilon}(x) \right)e_k \otimes e_l \right]
    \psi^{\varepsilon,z}_i e_i, \sigma_3 \eta \right\rangle\\ \nonumber
  &=  \left\langle  \left( \Div (\rho(x)
    \Phi_{kl}^{\varepsilon}(x))- \nabla \rho(x) \cdot
    \Phi_{kl}^{\varepsilon}(x) \right)
    \psi^{\varepsilon,z}_l e_k, \overline{\eta_1} e_1 -
    \overline{\eta_2} e_2 \right\rangle\\ \nonumber
  &=  \sum_{k,l=1}^2 (-1)^{k-1} \int_{\RR^2} \left( \Div (\rho(x)
    \Phi_{kl}^{\varepsilon}(x))- \nabla \rho(x) \cdot
    \Phi_{kl}^{\varepsilon}(x) \right)
    \psi^{\varepsilon,z}_l \overline{\eta_k} \di x\\ \nonumber
  &=  \sum_{k,l=1}^2 (-1)^k \left\{ \int_{\RR^2}  \rho(x)
    \Phi_{kl}^{\varepsilon}(x)
    \nabla \left( \psi^{\varepsilon,z}_l \overline{\eta_k}  \right)\di x
    + \int_{\RR^2} \nabla \rho(x) \cdot
    \Phi_{kl}^{\varepsilon}(x)
    \psi^{\varepsilon,z}_l \overline{\eta_k}\di x
     \right\}\\ \nonumber
&=\sum_{k,l=1}^2 (-1)^k \left\{ \int_{\RR^2}  \rho(x)
    \Phi_{kl}^{\varepsilon}(x)
    \left( \overline{\eta_k} \nabla \psi^{\varepsilon,z}_l  +
  \psi^{\varepsilon,z} \nabla  \overline{\eta_k}\right)\di x
    \right.\\
  &\qquad \left.
    + \int_{\RR^2} \nabla \rho(x) \cdot
    \Phi_{kl}^{\varepsilon}(x)
    \psi^{\varepsilon,z}_l \overline{\eta_k}\di x
    \right\}.
\end{align}

Thus, multiplying \eqref{eq:7} by
$\sigma_3 \eta$,  and integrating by parts, we obtain
\begin{align}
\label{eq:27}
  \begin{split}
    &\left\langle \psi^{\varepsilon,z}, (D \cdot \sigma + m \sigma_3
  -z ) \sigma_3\eta\right\rangle
-\left\langle \beta
      \nabla \sigma_3 \psi^{\varepsilon,z}, \nabla \sigma_3\eta
  \right\rangle \\
  &+ \sum_{k,l=1}^2 (-1)^k \left\{ \int_{\RR^2}  \rho(x)
    \Phi_{kl}^{\varepsilon}(x)
    \left( \overline{\eta_k} \nabla \psi^{\varepsilon,z}_l  +
  \psi^{\varepsilon,z}_l \nabla  \overline{\eta_k}\right)\di x
    \right.\\
  &\qquad \left.
    + \int_{\RR^2} \nabla \rho(x) \cdot
    \Phi_{kl}^{\varepsilon}(x)
    \psi^{\varepsilon,z}_l \overline{\eta_k}\di x
     \right\}\\
    &=\left\langle f,\sigma_3\eta \right\rangle.
  \end{split}
\end{align}
Therefore,
\begin{align}
\label{eq:46}
  \nonumber
  \abs{\left\langle \beta
  \nabla \psi^{\varepsilon,z}, \nabla \eta
  \right\rangle}
  &\le \abs{\left\langle f,\sigma_3 \eta \right\rangle} + \abs{\left\langle \psi^{\varepsilon,z}, (D \cdot \sigma + m \sigma_3
    -z ) \sigma_3\eta\right\rangle
    }\\\nonumber
  &\qquad +  \left\{ \abs{\int_{\RR^2}  \rho(x)
    \Phi_{kl}^{\varepsilon}(x)
    \left( \overline{\eta_k} \nabla \psi^{\varepsilon,z}_l  +
    \psi^{\varepsilon,z}_l \nabla  \overline{\eta_k}\right)\di x}
    \right.\\\nonumber
  &\qquad \qquad\left.
    + \abs{\int_{\RR^2} \nabla \rho(x) \cdot
    \Phi_{kl}^{\varepsilon}(x)
    \psi^{\varepsilon,z}_l \overline{\eta_k}\di x}
    \right\}\\\nonumber
  &\le \norm{f} \norm{\eta} + (2 + \abs{m} + \abs{z}) \norm{\psi^{\varepsilon,z}}
    \norm{\eta}_{H^1} \\\nonumber
  &\qquad + C\sum_{k,l=1}^2 \norm{\rho}_{L^{\infty}}
    \norm{\Phi_{kl}^{\varepsilon}}_{L^{\infty}} \left( \norm{\eta}
    \norm{\nabla \psi^{\varepsilon,z}} + \norm{\nabla \eta}
    \norm{\psi^{\varepsilon,z}} \right) \\\nonumber
  &\qquad + C \sum_{k,l=1}^2\norm{\nabla
    \rho}_{L^{\infty}} \norm{\Phi_{kl}^{\varepsilon}}_{L^{\infty}}
    \norm{\psi^{\varepsilon,z}} \norm{\eta}\\\nonumber
  &\le \norm{f} \norm{\eta} + (2 + \abs{m} + \abs{z}) \norm{\psi^{\varepsilon,z}}
    \norm{\eta}_{H^1} \\
  &\qquad +C \norm{\rho}_{L^{\infty}}
    \norm{W}_{C^{0,\alpha}} \left( \norm{\eta}
    \norm{\nabla \psi^{\varepsilon,z}} + \norm{\nabla \eta}
    \norm{\psi^{\varepsilon,z}} \right) \\\nonumber
  &\qquad + C\norm{\nabla
    \rho}_{L^{\infty}} \norm{W}_{C^{0,\alpha}}
    \norm{\psi^{\varepsilon,z}} \norm{\eta},
\end{align}
where in the last step we used the Schauder estimate
\cite{gilbargEllipticPartialDifferential2001} for \eqref{eq:6} to
obtain $\norm{\Phi_{kl}^{\varepsilon}}_{L^{\infty}} =
\norm{\Phi_{kl}}_{L^{\infty}} \le C \norm{W_{kl}}_{C^{0,\alpha}} \le C
\norm{W}_{C^{0,\alpha}}.$
 Moreover, by density, we can take $\eta = \psi^{\varepsilon,z}$.
Thus
\eqref{eq:46} becomes
\begin{align*}
  \begin{split}
    \abs{\beta} \norm{\nabla \psi^{\varepsilon,z}}^2
    &\le  \norm{f} \norm{\psi^{\varepsilon,z}} + (2 + \abs{m} + \abs{z}) \norm{\psi^{\varepsilon,z}}
      \norm{\psi^{\varepsilon,z}}_{H^1} \\
    &\qquad +C \norm{\rho}_{L^{\infty}}
      \norm{W}_{C^{0,\alpha}}  \norm{\nabla \psi^{\varepsilon,z}}
      \norm{\psi^{\varepsilon,z}}  \\
    &\qquad + C\norm{\nabla
      \rho}_{L^{\infty}} \norm{W}_{C^{0,\alpha}}
      \norm{\psi^{\varepsilon,z}} \norm{\psi^{\varepsilon,z}},\\
    &\le C \left( \abs{\Im z}^{-1} \norm{f}^2 + (2+\abs{m}+ \abs{z})\abs{\Im
      z}^{-1} \norm{f} (\abs{\Im z}^{-1} \norm{f} + \norm{\nabla
      \psi^{\varepsilon,z}})  \right.\\
    &\qquad \left. +
      \norm{\rho}_{L^{\infty}}\norm{W}_{C^{0,\alpha}} \norm{f} \norm{\nabla
      \psi^{\varepsilon,z}} + \norm{\rho}_{L^{\infty}}
      \norm{W}_{C^{0,\alpha}} \norm{f}^2
      \right)\\
    &\le C \left(  \abs{\Im z}^{-1} \norm{f}^2 + (2+\abs{m}+ \abs{z})\abs{\Im
      z}^{-2} \norm{f}^2 + \norm{\rho}_{L^{\infty}} \norm{W}_{C^{0,\alpha}}
      \norm{f}^2\right)\\
    &\qquad + C \left[ (2+\abs{m}+\abs{z})\abs{\Im z}^{-1} +
      \norm{\rho}_{L^{\infty}} \norm{W}_{C^{0,\alpha}} \right] \norm{f} \norm{\nabla \psi^{\varepsilon,z}}.
  \end{split}
\end{align*}
By Young's inequality,
\begin{align*}
  &C \left[ (2+\abs{m}+\abs{z})\abs{\Im z}^{-1} +
      \norm{\rho}_{L^{\infty}} \norm{W}_{C^{0,\alpha}} \right] \norm{f}
    \norm{\nabla \psi^{\varepsilon,z}}\\
  &\quad\le
    \frac{\abs{\beta}}{2} \norm{\nabla\psi^{\varepsilon,z}}^2 + \frac{1}{2\abs{\beta}}{C^2 \left[ (2+\abs{m}+\abs{z})\abs{\Im z}^{-1} +
      \norm{\rho}_{L^{\infty}} \norm{W}_{C^{0,\alpha}} \right]^2} \norm{f}^2.
\end{align*}
Thus,
\begin{align*}
  \frac{\abs{\beta}}{2} \norm{\nabla \psi^{\varepsilon,z}}^2
  &\le C \left( \abs{\Im z}^{-1} + (2 + \abs{m} + \abs{z})\abs{\Im z}^{-2} +
    \frac{1}{\abs{\beta}}(2 + \abs{m} + \abs{z})\abs{\Im z}^{-1} \right.\\
  &\qquad \left. +\frac{1}{\abs{\beta}} \norm{\rho}_{L^{\infty}}
    \norm{W}_{C^{0,\alpha}} + \frac{1}{\abs{\beta}} \norm{\rho}_{L^{\infty}}^2
    \norm{W}_{C^{0,\alpha}}^2\right) \norm{f}^2\ \le C \abs{\Im z}^{-2} \norm{f}^2,
\end{align*}
whenever $\Re z \in [-\lambda,\lambda]$ and
$\Im z \in [-\gamma, \gamma] \setminus \left\{ 0 \right\}$ with
$\gamma = \gamma (\lambda, m, \beta, \norm{\rho}_{L^{\infty}},
\norm{W}_{C^{0,\alpha}}) > 0$ small enough. Therefore,
\begin{align}
\label{eq:47}
  \norm{\nabla \psi^{\varepsilon,z}}
  \le C \abs{\Im z}^{-1}\norm{f}.
\end{align}
From \eqref{eq:45} and \eqref{eq:47}, we obtain \eqref{eq:9}.

  \item From now on, we suppress the dependence on $z $ to lighten
    the notation. By \eqref{eq:9} and \cite{allaireHomogenizationTwoscaleConvergence1992,nguetsengGeneralConvergenceResult1989,visintinTwoscaleCalculus2006}, there exist $\psi^s_1, \psi^s_2$ in
$H^1(\RR^2;\CC)$ and $\psi^f_1, \psi^f_2$ in $L^2(\RR^2;
H^1_{\per}(Y;\CC^2)/\CC)$ such that (up to a subsequence), we have the two-scale convergence
\begin{align}
\label{eq:11}
  \begin{split}
    \psi^{\varepsilon}
    =
    \begin{pmatrix}
      \psi^{\varepsilon}_1\\
      \psi^{\varepsilon}_2
    \end{pmatrix}
    &\wcv[2]
    \begin{pmatrix}
      \psi^s_1\\
      \psi^s_2
    \end{pmatrix} \quad \text{ and } \quad
    \nabla \psi^{\varepsilon}
    =
      \begin{pmatrix}
        \nabla \psi^{\varepsilon}_1\\
        \nabla \psi^{\varepsilon}_2
      \end{pmatrix}
    \wcv[2]
      \begin{pmatrix}
        \nabla \psi^s_1 + \nabla_y \psi^f_1\\
        \nabla \psi^s_2 + \nabla_y \psi^f_2
      \end{pmatrix}.
  \end{split}
\end{align}

\item Suppose $\eta = \eta^s + \varepsilon \eta^f$ with $\eta^s =
(\eta^s_1, \eta^s_2)^{\top} \in \calD(\RR^2; \CC^2)$ and $\eta^f =
(\eta^f_1, \eta^f_2)^{\top} \in \calD (\RR^2;
C_{\per}^{\infty}(Y;\CC^2))$.  Then \eqref{eq:7} implies
\begin{align}
\label{eq:8}
\left\langle (\opH^{\varepsilon}-z) \psi^{\varepsilon}, \eta
  \right\rangle
  = \left\langle f,\eta \right\rangle,
\end{align}
or equivalently,
\begin{align}
\label{eq:10}
J_1 + J_2 = \left\langle f,\eta \right\rangle,
\end{align}
where
\begin{align*}
  J_1
  &\coloneqq \left\langle (D \cdot \sigma + (m+\beta \Delta) \sigma_3
    -z ) \psi^{\varepsilon}, \eta\right\rangle,\\
  J_2
  &\coloneqq \left\langle U^{\varepsilon} \psi^{\varepsilon}, \eta \right\rangle.
\end{align*}
\item We now compute $\lim_{\varepsilon \to 0} J_1$. Using integration by parts, \eqref{eq:11} leads to
\begin{align*}
  \begin{split}
    J_1
    &= \left\langle (D \cdot \sigma + (m+\beta \Delta) \sigma_3
      -z ) \psi^{\varepsilon}, \eta\right\rangle\\
    &= \left\langle (D \cdot \sigma + m \sigma_3
    -z ) \psi^{\varepsilon}, \eta\right\rangle +
      \left\langle \beta \Delta \sigma_3 \psi^{\varepsilon}, \eta
      \right\rangle\\
    &= \left\langle \left[ -i \partial_{x_1} \sigma_1 - i
      \partial_{x_2} \sigma_2 + m \sigma_3 - z  \right]
      \psi^{\varepsilon}, \eta \right\rangle -\left\langle \beta
      \nabla \sigma_3 \psi^{\varepsilon}, \nabla \eta \right\rangle\\
    &\cv[\varepsilon \to 0]
    \left\langle \left( D \cdot \sigma \right) \psi^s + \left( D_y
      \cdot \sigma \right) \psi^f + m \sigma_3 \psi^s - z \psi^s,
      \eta^s \right\rangle \\
    &\qquad \qquad - \left\langle \beta \left[ (\nabla \sigma_3) \psi^s
      +(\nabla_y \sigma_3) \psi^f \right], \nabla \eta^s + \nabla_y \eta^f \right\rangle.
  \end{split}
\end{align*}
\item To compute $\lim_{\varepsilon \to 0} J_2,$ we write
\begin{align*}
  J_2
  &= \left\langle U^{\varepsilon} \psi^{\varepsilon}, \eta
    \right\rangle\\
  &=\left\langle \frac{1}{\varepsilon} \rho (x) W \left(
    \frac{x}{\varepsilon} \right) \psi^{\varepsilon}, \eta
    \right\rangle\\
  &= \left\langle \left[ \left( \Div (\rho(x)
    \Phi_{kl}^{\varepsilon}(x))- \nabla \rho(x) \cdot
    \Phi_{kl}^{\varepsilon}(x) \right)e_k \otimes e_l \right]
    \psi^{\varepsilon}_i e_i, \eta \right\rangle\\
  &=  \left\langle  \left( \Div (\rho(x)
    \Phi_{kl}^{\varepsilon}(x))- \nabla \rho(x) \cdot
    \Phi_{kl}^{\varepsilon}(x) \right)
    \psi^{\varepsilon}_l e_k, \overline{\eta_j} e_j\right\rangle\\
  &=  \int_{\RR^2} \left( \Div (\rho(x)
    \Phi_{kl}^{\varepsilon}(x))- \nabla \rho(x) \cdot
    \Phi_{kl}^{\varepsilon}(x) \right)
    \psi^{\varepsilon}_l \overline{\eta_k} \di x\\
  &=  -  \int_{\RR^2}  \rho(x)
    \Phi_{kl}^{\varepsilon}(x)
    \nabla \left( \psi^{\varepsilon}_l \overline{\eta_k}  \right)\di x
    - \int_{\RR^2} \nabla \rho(x) \cdot
    \Phi_{kl}^{\varepsilon}(x)
    \psi^{\varepsilon}_l \overline{\eta_k}\di x
     \\
&=- \int_{\RR^2}  \rho(x)
    \Phi_{kl}^{\varepsilon}(x)
    \left( \overline{\eta_k} \nabla \psi^{\varepsilon}_l  +
  \psi^{\varepsilon}_l \nabla  \overline{\eta_k}\right)\di x
    - \int_{\RR^2} \nabla \rho(x) \cdot
    \Phi_{kl}^{\varepsilon}(x)
    \psi^{\varepsilon}_l \overline{\eta_k}\di x
    \\
  &=- \int_{\RR^2}  \rho(x)
    \Phi_{kl} \left( \frac{x}{\varepsilon} \right)
    \left( \overline{\eta_k^s + \varepsilon \eta_k^f} \nabla \psi^{\varepsilon}_l  +
  \psi^{\varepsilon}_l \nabla  \overline{\eta_k^s + \varepsilon \eta_k^f}\right)\di x
   \\
  &\qquad
    - \int_{\RR^2} \nabla \rho(x) \cdot
    \Phi_{kl}\left( \frac{x}{\varepsilon} \right)
    \psi^{\varepsilon}_l \overline{\eta_k^s + \varepsilon \eta_k^f}\di x
  \\
  &\cv[\varepsilon \to 0]
    - \int_{\RR^2 \times Y}  \rho(x)
    \Phi_{kl} \left( y \right)
    \left( \overline{\eta_k^s} \left( \nabla \psi_l^s + \nabla_y \psi_l^f \right)  +
  \psi_l^s \left( \nabla \overline{\eta_k^s} + \nabla_y
    \overline{\eta_k^f} \right)\right)\di x \di y
   \\
  &\qquad
    - \int_{\RR^2\times Y} \nabla \rho(x) \cdot
    \Phi_{kl}\left( y \right)
    \psi_l^s \overline{\eta_k^s}\di x \di y,
\end{align*}
where we used \eqref{eq:11} in the last convergence.
\item Therefore, passing to the limit $\varepsilon \to 0$ in \eqref{eq:10}, we
  obtain
\begin{align}
\label{eq:13}
  \begin{split}
    &\left\langle \left( D \cdot \sigma \right) \psi^s + \left( D_y
      \cdot \sigma \right) \psi^f + m \sigma_3 \psi^s - z \psi^s,
      \eta^s \right\rangle \\
    &\qquad \qquad - \left\langle \beta \left[ (\nabla \sigma_3) \psi^s
      +(\nabla_y \sigma_3) \psi^f \right], \nabla \eta^s + \nabla_y
      \eta^f \right\rangle\\
    &\qquad \qquad - \int_{\RR^2 \times Y}  \rho(x)
    \Phi_{kl} \left( y \right)
    \left( \overline{\eta_k^s} \left( \nabla \psi_l^s + \nabla_y \psi_l^f \right)  +
  \psi_l^s \left( \nabla \overline{\eta_k^s} + \nabla_y
      \overline{\eta_k^f} \right)\right)\di x \di y
   \\
  &\qquad \qquad \qquad
    - \int_{\RR^2 \times Y} \nabla \rho(x) \cdot
    \Phi_{kl}\left( y \right)
    \psi_l^s \overline{\eta_k^s}\di x \di y\ \ =\  \left\langle f, \eta^s \right\rangle.
  \end{split}
\end{align}
\item We derive the cell problems. In \eqref{eq:13}, let $\eta^s = 0$,
  then
  \begin{align*}
    &-\left\langle \beta \left[ (\nabla \sigma_3) \psi^s
      +(\nabla_y \sigma_3) \psi^f \right],  \nabla_y
      \eta^f \right\rangle%
      - \int_{\RR^2 \times Y} \rho(x) \Phi_{kl}(y) \psi_l^s
      \nabla_y \overline{\eta_k^f}  \di x
      \di y = 0,
  \end{align*}
  or equivalently,
\begin{align}
  \label{eq:14}
  \left\langle \beta \left[ (\nabla \sigma_3) \psi^s
      +(\nabla_y \sigma_3) \psi^f \right]
      + \rho (x) \left\{ \Phi_{kl}(y) e_k \otimes e_l  \right\}\psi^s_j e_j , \nabla_y \eta^f
      \right\rangle = 0.
\end{align}
Substuting the ansatz
\begin{align}
\label{eq:15}
  \psi^f(x,y)
  \coloneqq \rho(x)\sigma_3 \left\{ T_{kl}(y) e_k \otimes e_l  \right\} \psi^s_j
  e_j
  = \rho(x) \sigma_3
  \begin{pmatrix}
    T_{11}(y) & T_{12}(y)\\
    T_{21}(y) & T_{22} (y)
  \end{pmatrix}
  \psi^s(x),
\end{align}
with $T_{kl} \in H^1_{\per}(Y; \CC), ~1 \le k,l \le 2,
$ into \eqref{eq:14}, and noticing that $\left\langle \beta (\nabla
  \sigma_3) \psi^s, \nabla_y \eta^f \right\rangle = 0$ by integration
by parts and periodicity, we obtain
\begin{align*}
\left\langle \rho(x) \left\{ \left(\beta \nabla_yT_{kl}(y) + \Phi_{kl}(y)  \right)
  e_k \otimes e_l \right\} \psi_j^s e_j, \nabla_y \eta^f
  \right\rangle = 0,
\end{align*}
holds for any $\eta^f \in \calD(\RR^2;
C^{\infty}_{\per}(Y;\CC^2))$. In the above equation, choose
$\eta^f (x,y) = \rho_0(x) \theta (y)$ where
$\rho_0 \in \calD(\RR^2;\RR)$ and $\theta \in C_{\per}^{\infty}(Y;\CC^2)$. Then, Fubini's theorem implies
\begin{align*}
\int_{\RR^2} \rho(x) \rho_0(x) \psi_l^s (x) \di x \int_Y \left(
 \beta \nabla_yT_{kl}(y) + \Phi_{kl}(y) \right) \nabla_y
  \overline{\theta_k} \di y  = 0.
\end{align*}
Recall from \eqref{eq:5} that $\Div_y \Phi_{kl} = W_{kl}$. We choose
\begin{align}
\label{eq:31}
T_{kl} \in H^1_{\per}(Y; \CC), \quad\beta \Delta_{yy} T_{kl} + W_{kl}
  = 0, \qquad 1 \le k,l\le 2.
\end{align}
Clearly, \eqref{eq:31} is well-posed.

\item We now derive the effective problem. In \eqref{eq:13}, let $\eta^f =
  0$. Then,
\begin{align}
\label{eq:18}
\begin{split}
    &\left\langle \left( D \cdot \sigma \right) \psi^s + \left( D_y
      \cdot \sigma \right) \psi^f + m \sigma_3 \psi^s - z \psi^s,
      \eta^s \right\rangle  - \left\langle \beta \left[ (\nabla \sigma_3) \psi^s
      +(\nabla_y \sigma_3) \psi^f \right], \nabla \eta^s\right\rangle\\
    &\qquad \qquad - \int_{\RR^2 \times Y}  \rho(x)
    \Phi_{kl} \left( y \right)
    \left( \overline{\eta_k^s} \left( \nabla \psi_l^s + \nabla_y \psi_l^f \right)  +
  \psi_l^s  \nabla \overline{\eta_k^s} \right)\di x \di y
   \\
  &\qquad \qquad \qquad
    - \int_{\RR^2 \times Y} \nabla \rho(x) \cdot
    \Phi_{kl}\left( y \right)
    \psi_l^s \overline{\eta_k^s}\di x \di y\ \ = \left\langle f, \eta^s \right\rangle.
  \end{split}
\end{align}
Note that
$\left\langle (D_y \cdot \sigma) \psi^f, \eta^s \right\rangle = 0$ by
\eqref{eq:15} and periodicity. Moreover, recall that
$\Phi_{kl}(y) = \nabla_y \phi_{kl}(y)$ for some $\phi_{kl} \in H^2_{\per}(Y)$, so
$\int_Y \Phi_{kl}(y) \di y =0$. Thus \eqref{eq:18} becomes
\begin{align}
\label{eq:19}
\begin{split}
    &\left\langle \left( D \cdot \sigma \right) \psi^s  + m \sigma_3 \psi^s - z \psi^s,
      \eta^s \right\rangle  - \left\langle \beta \left[ (\nabla \sigma_3) \psi^s
      +(\nabla_y \sigma_3) \psi^f \right], \nabla \eta^s\right\rangle\\
    &\qquad \qquad - \int_{\RR^2 \times Y}  \rho(x)
    \Phi_{kl} \left( y \right)
     \overline{\eta_k^s}  \nabla_y \psi_l^f\di x \di y\ \ = \left\langle f, \eta^s \right\rangle.
  \end{split}
\end{align}
We obtain by \eqref{eq:15}, integration by parts and periodicity that
\begin{align}
  \label{eq:20}
  \begin{split}
    &\left\langle \beta \left[ (\nabla \sigma_3) \psi^s
      +(\nabla_y \sigma_3) \psi^f \right], \nabla \eta^s\right\rangle\\
    &=\left\langle \beta \left[ (\nabla \sigma_3) \psi^s
      +(\nabla_y \sigma_3) \rho(x) \sigma_3
  T (y)
  \psi^s(x) \right], \nabla \eta^s\right\rangle\ \ = \left\langle \beta \left( \nabla \sigma_3 \right) \psi^s,
      \nabla \eta^s \right\rangle.
  \end{split}
\end{align}
Moreover,
\begin{align}
\label{eq:21}
  \nonumber
    &- \int_{\RR^2 \times Y}  \rho(x)
      \Phi_{kl} \left( y \right)
      \overline{\eta_k^s(x)}  \nabla_y \psi_l^f(x,y)\di x \di y\\
  \nonumber
    &=- \int_{\RR^2}  \rho(x)  \overline{\eta_k^s(x)}  \int_Y
      \Phi_{kl} \left( y \right)
      \nabla_y \psi_l^f(x,y)\di y \di x\\ \nonumber
    &= \int_{\RR^2} \rho(x)  \overline{\eta_k^s(x)} \int_Y
      \psi_l^f(x,y) \Div \Phi_{kl}(y) \di y \di x\\\nonumber
    &= \int_{\RR^2} \rho(x)  \overline{\eta_k^s(x)} \int_Y
      \psi_l^f(x,y) W_{kl}(y) \di y \di x\\\nonumber
    &= \int_{\RR^2} \rho(x)  \overline{\eta_k^s(x)} \int_Y
      \left( \psi^f(x,y) e_l \right) W_{kl}(y) \di y \di x\\\nonumber
    &= \int_{\RR^2} \rho(x)  \overline{\eta_k^s(x)} \int_Y
      \left(  \left( \rho(x)\sigma_3 \left\{ T_{nm}(y) e_n \otimes e_m  \right\} \psi^s_j(x)
      e_j \right) e_l \right) W_{kl}(y) \di y \di x\\\nonumber
    &= \int_{\RR^2} \rho^2(x)  \overline{\eta_k^s(x)} \int_Y
      \left(  \left( \sigma_3 \left\{ T_{nm}(y) \psi^s_m(x) e_n   \right\}
      \right) e_l \right) W_{kl}(y) \di y \di x\\\nonumber
    &= \int_{\RR^2} \rho^2(x)  \overline{\eta_k^s(x)} \int_Y
      \left(  \left( \sum_{n = 1}^2 (-1)^{n-1}  T_{nm}(y) \psi^s_m(x) e_n
      \right) e_l \right) W_{kl}(y) \di y \di x\\\nonumber
    &= \int_{\RR^2} \rho^2(x)  \overline{\eta_k^s(x)} \int_Y
         \sum_{n = 1}^2 (-1)^{n-1}  T_{nm}(y) \psi^s_m(x)
      W_{kn}(y) \di y \di x\\\nonumber
    &= \int_{\RR^2} \rho^2(x)  \overline{\eta_k^s(x)} \psi^s_m(x)   \int_Y
         \sum_{n = 1}^2 (-1)^{n-1}  T_{nm}(y)
      W_{kn}(y) \di y \di x\\
    &= \beta \int_{\RR^2} \rho^2(x) \tau \psi^s (x)\cdot
      \overline{\eta^s(x)}  \di x,
\end{align}
with $\tau \in \CC^{2 \times 2}$ such that
\begin{align*}
\tau_{km} \coloneqq \frac{1}{\beta}\int_Y \sum_{n=1}^2 (-1)^{n-1} W_{kn} (y)
  T_{nm}(y)   \di y.
\end{align*}
Clearly, $\tau  = \frac{1}{\beta}\int_Y W (y) \sigma_3 T(y) \di y$.
From \eqref{eq:31} and integration by parts, we obtain
\begin{align*}
  \tau_{km}
  &= \frac{1}{\beta}\int_Y \sum_{n=1}^2 (-1)^{n-1} (-\beta \Delta_{yy}
  T_{kn}(y)) T_{nm} (y) \di y \\
  &=  \int_Y \sum_{n=1}^2 (-1)^{n-1} \nabla_y T_{kn}(y) \nabla_y
    T_{nm} (y) \di y,
\end{align*}
which implies \eqref{eq:30}.

From \eqref{eq:19}, \eqref{eq:20}, and \eqref{eq:21},
\begin{align*}
  \begin{split}
    &\left\langle \left( D \cdot \sigma \right) \psi^s  + m \sigma_3 \psi^s - z \psi^s,
      \eta^s \right\rangle - \left\langle \beta \left( \nabla \sigma_3 \right) \psi^s,
      \nabla \eta^s \right\rangle + \left\langle \beta \rho^2 \tau \psi^s, \eta^s
      \right\rangle%
      = \left\langle f, \eta^s \right\rangle,
  \end{split}
\end{align*}
holds for all $\eta^s \in \calD(\RR^2;\CC^2)$. Therefore,
\begin{align*}
  \left[ D \cdot \sigma + (m+\beta \Delta) \sigma_3 - z + \beta \rho^2
  \tau\right] \psi^s
  = f,
\end{align*}
or equivalently,
\begin{align*}
(\opH^0 - z) \psi^s = f.
\end{align*}
By a similar argument as in \lemref{lem:main-results-4}, we have
$\opH^0$ is a self-adjoint operator in $L^2(\RR^2;\CC^2)$ with domain
$\dom (\opH^0) = H^2(\RR^2;\CC^2)$. Therefore, the above homogenized
equation has a unique solution $\psi^s \in H^2(\RR^2;\CC^2)$, and
thus, the full sequence $\psi^{\varepsilon}$ converges to $\psi^s.$
\end{enumerate}
\end{proof}

\begin{proof}[Proof of \lemref{lem:limit-absorpt-princ-1}]
By \eqref{eq:9},
\begin{align}
  \label{eq:50}
  \norm{\psi^s}_{H^1}
  \le \limsup_{\varepsilon \to 0} \norm{\psi^{\varepsilon,z}}_{H^1}
  \le C \abs{\Im z}^{-1} \norm{f}_{L^2}.
\end{align}
Taking Fourier transform of $(\opH^0-z)\psi^{s,z} = f$, we obtain
\begin{align*}
\left( \xi \cdot \sigma + (m-\beta \abs{\xi}^2) \sigma_3 \right)\widehat{\psi^{s,z}} + \beta  \ftf \left( \rho^2 \tau \psi^{s,z} \right)- z \widehat{\psi^{s,z}}
  = \hat{f}.
\end{align*}
Multiplying both side by $\sigma_3$ and rearranging terms, we get
\begin{align}
\label{eq:49}
  \abs{\beta} \abs{\xi}^2 \abs{\widehat{\psi^{s,z}}}
  \le C \left( \abs{\hat{f}} + \abs{z + m} \abs{\widehat{\psi^{s,z}}}
  + \abs{\xi} \abs{\widehat{\psi^{s,z}}} + \abs{\beta}  \abs{\ftf \left(
  \rho^2 \tau \psi^{s,z}\right)} \right).
\end{align}
Square both sides, use Cauchy-Schwarz, integrate over $\RR^2$, then apply Plancherel theorem, to conclude that:
\begin{align*}
  \beta^2 \norm{\abs{\xi}^2 \abs{\widehat{\psi^{s,z}}}}^2
  &\le C \left( \norm{\hat{f}}^2 + \abs{z +m}^2
    \norm{\widehat{\psi^{s,z}}}^2 + \norm{\abs{\xi}
    \abs{\widehat{\psi^{s,z}}}}^2 + \beta^2  \norm{\ftf \left(
    \rho^2 \tau \psi^{s,z} \right)}^2 \right)\\
  &\le C \left( \norm{{f}}^2 + \abs{z +m}^2
    \norm{{\psi^{s,z}}}^2 + \norm{
    {\psi^{s,z}}}_{H^1}^2 + \beta^2  \norm{
    \rho^2 \tau \psi^{s,z} }^2 \right)\\
  &\le C \left( \norm{{f}}^2 + \abs{z +m}^2
    \norm{{\psi^{s,z}}}^2 + \norm{
    {\psi^{s,z}}}_{H^1}^2 + \beta^2  \norm{
    \rho^2  \psi^{s,z} }^2 \right),
\end{align*}
where, in the last step, we used the bound
$\abs{\tau} \le C ( \beta, \norm{W}_{C^{0,\alpha}} )$, which follows from \eqref{eq:25}, \eqref{eq:30} and
elliptic regularity \cite{gilbargEllipticPartialDifferential2001}.
Now the fact that $\rho \in L^{\infty}$ and \eqref{eq:50} imply
\begin{align*}
 \beta^2 \norm{\abs{\xi}^2 \abs{\widehat{\psi^{s,z}}}}^2
  \le C \abs{\Im z}^{-2} \norm{f}^2,
\end{align*}
which leads to \eqref{eq:47b}.
\end{proof}

\begin{proof}[Proof of \cref{thm:limit-absorpt-princ-4}]
  \begin{enumerate}[wide]
  \item We first consider the case
    $z \in [-\lambda,\lambda] + i ([-\gamma,\gamma] \setminus \left\{
      0 \right\})$ where $\lambda, \gamma$ are defined in
    \cref{thm:main-results-2}, so that
    \lemref{lem:limit-absorpt-princ-1} can be used freely. Without loss of generality, $\gamma\le1$, and we also assume $\eps\le1$. Since the
    dependence of solutions on $z$ is not important for the following
    argument, we drop the superscript $^z$ to lighten the
    notation. Let $y = \frac{x}{\varepsilon}$. Then
    $\partial_i = \partial_{x_i} + \frac{1}{\varepsilon}
    \partial_{y_i}$.  Define
\begin{align*}
  \opK^{\varepsilon}
  &\coloneqq \opH^{\varepsilon} - z,\\
  \opK^0
  &\coloneqq \beta \Delta_{yy}\sigma_3,\\
  \opK^1
  &\coloneqq
    D_y \cdot \sigma + 2\beta \nabla_x \cdot \nabla_y \sigma_3 +
    \rho(x)W(y),\\
  \opK^2
  &\coloneqq D_x \cdot \sigma + (m + \beta\Delta_{xx}) \sigma_3 - z,
\end{align*}
so that $\opK^\eps\bigl(u(x,\tfrac x\eps)\bigr)=\bigl[(\eps^{-2}\opK^0+\eps^{-1}\opK^1+\opK^2)u\bigr](x,\tfrac x\eps)$ for functions $u(x,y)$ periodic in $y$ and smooth enough. For $f\in L^2(\RR^2;\CC^2)$, let $\psi^s=(\opH^0-z)^{-1}f\in H^2(\RR^2;\CC^2)$. By \lemref{lem:limit-absorpt-princ-1},
\begin{align}
\label{eq:psisH2}
\norm{\psi^s}_{H^2}\le C\abs{\Im z}^{-1}\norm{f}.
\end{align}
We use the following smoothing at the oscillation scale, in the spirit of \cite{ZP05}. Let $S_\eps$ be the Fourier multiplier by $\mathbf 1_{\{\abs{\xi}\le1/\eps\}}$. Since $\langle\xi\rangle^{-1}\le\eps$ for $\abs\xi>1/\eps$ and $\langle\xi\rangle\le\sqrt2\,\eps^{-1}$ for $\abs\xi\le1/\eps$,
\begin{align}
\label{eq:Seps}
\norm{(I-S_\eps)h}\le\eps\norm{h}_{H^1},\qquad \norm{S_\eps h}_{H^1}\le\norm{h}_{H^1},\qquad \norm{S_\eps h}_{H^2}\le\sqrt2\,\eps^{-1}\norm{h}_{H^1} .
\end{align}

\item Recall from \eqref{eq:31} that
  \begin{align}
    \label{eq:38}
\psi^f = \rho(x) \sigma_3  T(y) \psi^s(x).
\end{align}
Since $\opK^0\psi^s=0$ and $\opK^0\psi^f=\rho\,\beta(\Delta_{yy}T)\psi^s=-\rho W\psi^s=-\opK^1\psi^s$ by \eqref{eq:25}, the terms of orders $\eps^{-2}$ and $\eps^{-1}$ in $\opK^\eps(\psi^s+\eps\psi^f)$ vanish. At order $\eps^0$, by \eqref{eq:23}--\eqref{eq:30}, $f-\opK^2\psi^s=\beta\rho^2\tau\psi^s$, and
\begin{align}
\label{eq:33}
\begin{split}
  \beta \rho^2(x) \tau \psi^s(x) - \opK^1\psi^f
  &= \beta \rho^2(x) \tau \psi^s(x)\\
  &\qquad - \left( D_y \cdot \sigma  + 2\beta \nabla_x \cdot \nabla_y \sigma_3 +
    \rho(x)W(y) \right) \left[ \rho(x) \sigma_3 T(y) \psi^s(x) \right]\\
    &=\sum_{i=1}^4 k_i(y)h_i(x),
\end{split}
\end{align}
where $h_1=\rho^2\psi^s$, $h_2=\rho\psi^s$ and $h_{2+j}=\partial_{x_j}(\rho\psi^s)$, $j=1,2$, so that
\begin{align}
\label{eq:34}
 \norm{h_i}_{H^1}
  \le C (1+\norm{\rho}_{W^{2,\infty}})^2\norm{\psi^s}_{H^2},
\end{align}
and $k_1=\beta\tau-W\sigma_3T$, $k_2=-(D_y\cdot\sigma)\sigma_3T$ and $k_{2+j}=-2\beta\,\partial_{y_j}T$ are matrix-valued functions in $C^{0,\alpha}_{\per}(Y)$ with $\int_Y k_i\di y=0$, by the definition \eqref{eq:30} of $\tau$ and the periodicity of $T$; here $T\in C^{2,\alpha}_{\per}(Y)$ by Schauder estimates \cite{gilbargEllipticPartialDifferential2001}. The exact formulae of $h_i,
k_i$ are not important. We solve the cell problems $\opK^0\psi^{r,i}=k_i$, with $\psi^{r,i}\in C^{2,\alpha}_{\per}(Y;\CC^{2\times2})$, and, to avoid derivatives of $h_i$ of order higher than those controlled by \eqref{eq:34}, we define the second corrector with smoothed coefficients,
\begin{align}
  \label{eq:37}
  \psi^r(x,y) = \sum_{i=1}^4 \psi^{r,i} (y)\,(S_\eps h_i)(x),\qquad\mbox{so that}\qquad \opK^0\psi^r=\sum_{i=1}^4k_i(y)\,(S_\eps h_i)(x).
\end{align}
\item Let
\begin{align*}
  Z^{\varepsilon}(x)
  \coloneqq \psi^{\varepsilon} - \left( \psi^s + \varepsilon \psi^f +
  \varepsilon^2 \psi^r \right) \left( x, \frac{x}{\varepsilon} \right),
\end{align*}
which belongs to $H^2(\RR^2;\CC^2)=\dom(\opH^\eps)$. By the second step and \eqref{eq:37},
\begin{align*}
  \begin{split}
    \opK^{\varepsilon} Z^{\varepsilon}
  &= f - \left( \frac{1}{\varepsilon^2} \opK^0 +
      \frac{1}{\varepsilon}\opK^1 + \opK^2 \right) \left( \psi^s +
      \varepsilon \psi^f + \varepsilon^2 \psi^r \right)\\
    &= \left[ f - (\opK^0 \psi^r + \opK^1 \psi^f + \opK^2 \psi^s)
      \right] - \varepsilon \left( \opK^1 \psi^r + \opK^2 \psi^f
      \right) - \varepsilon^2 \opK^2 \psi^r\\
    &=\sum_{i=1}^4k_i(y)\bigl((I-S_\eps)h_i\bigr)(x)- \varepsilon \left( \opK^1 \psi^r + \opK^2 \psi^f
      \right) - \varepsilon^2 \opK^2 \psi^r,
  \end{split}
\end{align*}
all functions of $(x,y)$ being evaluated at $y=x/\eps$. We bound the three terms in $L^2$, using $\norm{k(x/\eps)h(x)}\le\norm{k}_{L^\infty}\norm{h}$. By \eqref{eq:Seps} and \eqref{eq:34}, the first term is bounded by $C\eps\max_i\norm{h_i}_{H^1}\le C\eps\norm{\psi^s}_{H^2}$. The operator $\opK^1$ involves at most one derivative in $x$, and $\opK^2$ at most two, so that, by \eqref{eq:38}, \eqref{eq:Seps} and \eqref{eq:34},
\begin{align*}
\norm{\opK^1\psi^r}\le C\max_i\norm{S_\eps h_i}_{H^1}\le C\norm{\psi^s}_{H^2},\qquad
\norm{\opK^2\psi^f}\le C(1+\abs z)\norm{\psi^s}_{H^2},
\end{align*}
and
\begin{align*}
\eps\norm{\opK^2\psi^r}\le C(1+\abs z)\,\eps\max_i\norm{S_\eps h_i}_{H^2}\le C(1+\abs z)\max_i\norm{h_i}_{H^1}\le C(1+\abs z)\norm{\psi^s}_{H^2}.
\end{align*}
Here and below, $C$ depends only on $\lambda,  m, \beta, \norm{\rho}_{W^{2,\infty}}$ and $\norm{W}_{C^{0,\alpha}}$. Therefore, using \eqref{eq:psisH2} and $\abs{z}\le\lambda+1$,
\begin{align}
  \label{eq:36}
  \norm{\opK^{\varepsilon} Z^{\varepsilon}}\le C\eps(1+\abs z)\norm{\psi^s}_{H^2}\le C\eps\abs{\Im z}^{-1}\norm{f},
\end{align}
uniformly for $f\in L^2(\RR^2;\CC^2)$. Applying \eqref{eq:9} to $(\opH^\eps-z)Z^\eps=\opK^\eps Z^\eps$, we obtain
\begin{align}
\label{eq:53}
  \norm{Z^{\varepsilon}}_{H^1}
  \le C \abs{\Im z}^{-2} \varepsilon\norm{f}.
\end{align}
\item By \eqref{eq:38}, \eqref{eq:37} and \eqref{eq:Seps}, $\norm{\psi^f}\le C\norm{\psi^s}$ and $\norm{\psi^r}\le C\max_i\norm{h_i}\le C\norm{\psi^s}_{H^1}$, so that $\eps\norm{\psi^f}+\eps^2\norm{\psi^r}\le C\eps\abs{\Im z}^{-1}\norm{f}$ by \eqref{eq:psisH2}. With \eqref{eq:53} and $\abs{\Im z}^{-1}\le\abs{\Im z}^{-2}$, since $\gamma\le1$, this gives
\begin{align}
\label{eq:58}
\norm{(\opH^{\varepsilon}-z)^{-1} f- (\opH^{0} - z)^{-1}f}=\norm{\psi^{\varepsilon} - \psi^s} \le C \abs{\Im z}^{-2}\varepsilon\norm{f}
\end{align}
for every $f\in L^2(\RR^2;\CC^2)$. Therefore,
\begin{align}
\label{eq:59}
\norm{(\opH^{\varepsilon}  - z )^{-1} - (\opH^{0} - z
  )^{-1}}_{L^2 \to L^2}
  \le  C \varepsilon \abs{\Im z}^{-2},
\end{align}
for any $\varepsilon
\in (0,1]$ and $z \in [-\lambda,\lambda] + i \left(
  [-\gamma,\gamma]\setminus \left\{ 0 \right\} \right)$.
Norm-resolvent estimates for elliptic systems of arbitrary order, for smooth $\rho$, are proved in \cite[Theorem~2.1]{bal26} by a related smoothing argument.

\item We now consider $z \in [-\lambda, \lambda ] + i (\RR\setminus
  \left\{ 0 \right\})$. By \cite[Lemma
  2.6.1]{daviesSpectralTheoryDifferential1995} and \eqref{eq:59}, we
  obtain
\begin{align*}
&\norm{(\opH^{\varepsilon}  - z )^{-1} - (\opH^{0} - z
  )^{-1}}_{L^2 \to L^2}\\
&\quad= \gamma^{-1} \norm{(\gamma^{-1}\opH^{\varepsilon}  -\gamma^{-1} z )^{-1} - (\gamma^{-1}\opH^{0} - \gamma^{-1}z
  )^{-1}}_{L^2 \to L^2}\\
  &\quad\le \gamma^{-1} \frac{9 \left( 1 + \abs{\gamma^{-1} z}^2
    \right)}{\abs{\gamma^{-1} \Im z}^2} \norm{(\gamma^{-1}
    \opH^{\varepsilon} + i)^{-1} - (\gamma^{-1} \opH^0 + i)^{-1}}\\
  &\quad=  \frac{9 \left( \gamma^{2} + \abs{ z}^2
    \right)}{\abs{ \Im z}^2} \norm{(
    \opH^{\varepsilon} + i\gamma)^{-1} - ( \opH^0 + i\gamma)^{-1}}\\
  &\quad\le 9 \left( (\lambda^2 + \gamma^2)\abs{\Im z}^{-2} + 1\right)
    C \varepsilon \abs{\gamma}^{-2}\ \ \le \ C \left( 1+ \abs{\Im z}^{-2} \right) \varepsilon.
\end{align*}
\end{enumerate}
\end{proof}
\begin{remark}
  \label{sec:main-results-3}
  Estimate \eqref{eq:53}, together with $\norm{\eps^2\psi^r(\cdot,\tfrac\cdot\eps)}_{H^1}\le C\eps\norm{\psi^s}_{H^2}$, implies the rate of convergence
\begin{align}
\label{eq:41}
  \norm{\psi^{\varepsilon} - \psi^s - \varepsilon \psi^f\bigl(\cdot,\tfrac\cdot\eps\bigr)}_{H^1}
  \le C \abs{\Im z}^{-2}\varepsilon\norm{f}.
\end{align}
Note  that our problem does not have a boundary layer effect, so we
obtain the rate $O(\varepsilon)$ instead of
$O(\varepsilon^{\frac{1}{2}})$ as in classical results \cite{bensoussanAsymptoticAnalysisPeriodic2011}.
\end{remark}

\begin{proof}[Proof of \propref{prop:BI}, of \eqref{eq:interfaceindex}, and of \lemref{lem:index} for $T^0$]
The statements for $\opH^\infty$ are those of \cite{balContinuousBulkInterface2019}. For $\opH^0_B$, let $d$ be as in \eqref{eq:dxi}. By \cref{sec:topol-class-1}, $\inf_\xi\abs{d(\xi)}>\abs{\beta\tau_0}$, and $m_{\rm eff}\ne0$. We first deform $\beta\tau_0$ to $0$ along $s\beta\tau_0$, $s\in[0,1]$; the eigenvalues $s\beta\tau_0\pm\abs{d(\xi)}$ stay away from $0$, so that the gap at $0$ is preserved and the index of $F$ does not change. We then conjugate the bulk operator by the multiplication operator $e^{i\beta\tau_\perp\cdot x}$, which commutes with multiplication by $(x_1+ix_2)/\abs{x_1+ix_2}$ and therefore leaves the index of $F$ unchanged. The symbol becomes $\eta\cdot\sigma+\bigl(m_{\rm eff}+2\beta^2\tau_\perp\cdot\eta-\beta\abs{\eta}^2\bigr)\sigma_3$, with $\eta=\xi+\beta\tau_\perp$. Finally, we remove the term $2s\beta^2\tau_\perp\cdot\eta$ by letting $s$ go from $1$ to $0$: since the first two components $\eta$ vanish only at $\eta=0$, where the third one equals $m_{\rm eff}\ne0$, and since the symbol grows at infinity, the symbol is uniformly invertible along this deformation. At $s=0$, the operator is $D\cdot\sigma+(m_{\rm eff}+\beta\Delta)\sigma_3$, to which the formula of \cite{balContinuousBulkInterface2019} applies, and \eqref{eq:bulkindex} follows. By the bulk--edge correspondence of \cite{bal2022topological,QB22}, $2\pi\sigma_I[\opH^0]={\rm Index}\,T^0={\rm Index}\,F[\opH^0_B]-{\rm Index}\,F[\opH^\infty]$, which gives \eqref{eq:interfaceindex} and \lemref{lem:index}. Note that the conjugation is applied to the bulk operator only: applied to the interface operator $\opH^0$, it would also shift the upper bulk $\opH^\infty$.
\end{proof}

\begin{proof}[Proof of \cref{thm:stab}]
Using the Helffer-Sj\"ostrand formula (as used, e.g., in \cite{QB22}), we obtain
\begin{equation}\label{eq:HS}
   U(\opH^{\varepsilon}) - U(\opH^{l}) = S(\opH^{\varepsilon}) - S(\opH^{l}) = -\frac{1}{\pi} \int_{\CC} \bar\partial \tilde S(z) \Big( (z-\opH^{\varepsilon})^{-1} - (z-\opH^{l})^{-1}\Big) dz
\end{equation}
with $dz$ Lebesgue measure on $\CC$ and $\tilde S$ an almost analytic
extension of $S$ chosen so that $\bar\partial \tilde S(z)$ is
compactly supported in $[-\lambda,\lambda] + i \RR$ and such that $|\Im z|^{-N}|\bar\partial \tilde
S(z)|\leq C_N$ is bounded uniformly on that support for each $N\geq0$.
Here, the superscript $l$ stands for either limit $l=0$ or $l=\infty$.

From \cref{thm:limit-absorpt-princ-4}, we have
\[
  \norm{(z-\opH^{\varepsilon})^{-1} - (z-\opH^{0})^{-1} } \leq C (1+|\Im
  z|^{-2}) \varepsilon,\quad 0 < \varepsilon \leq \varepsilon_0'
\]
and from the resolvent identity (note that $\opH^{\varepsilon} =
\opH^{\infty} + \frac{1}{\varepsilon} \rho(x) W \left(
  \frac{x}{\varepsilon} \right)$) and the resolvent estimate,
\[
  \norm{(z-\opH^{\varepsilon})^{-1} - (z-\opH^{\infty})^{-1} } \leq C |\Im z|^{-2} \frac1\varepsilon ,\quad 0<  \varepsilon_1' \leq \varepsilon.
\]
Combined with the Helffer-Sj\"ostrand formula \eqref{eq:HS}, this implies that
\begin{align}
\label{eq:56}
\norm{ U(\opH^{\varepsilon}) - U(\opH^{0})} \leq C \varepsilon,\qquad  \norm{ U(\opH^{\varepsilon}) - U(\opH^{\infty})} \leq C \frac 1\varepsilon,
\end{align}
for $\varepsilon$ sufficiently small for the first estimate and sufficiently large for the second estimate.

From their definition,
\begin{align}
\label{eq:55}
T^{\varepsilon}- T^l= P(x) \left( U(\opH^{\varepsilon}) - U(\opH^l) \right) P(x).
\end{align}
Since the set $\calF(L^2(\RR^2;\CC^2),L^2(\RR^2;\CC^2))$ of Fredholm
operators from $L^2(\RR^2;\CC^2)$ to $L^2(\RR^2;\CC^2)$ is an open
set in the space of bounded operators
$\calB (L^2(\RR^2;\CC^2),L^2(\RR^2;\CC^2))$, we obtain from
\eqref{eq:56} and \eqref{eq:55} that $T^\varepsilon$ is Fredholm for
$\varepsilon$ sufficient large and $\varepsilon$ sufficient
small. Moreover, since the
$\Index \colon \calF(L^2(\RR^2;\CC^2),L^2(\RR^2;\CC^2)) \to \ZZ $ is a
continuous function, which is locally constant on the connected
components of $\calF(L^2(\RR^2;\CC^2),L^2(\RR^2;\CC^2))$
(cf. e.g. \cite[Section 5.2]{cheverryGuideSpectralTheory2021}), we obtain \eqref{eq:indexeps}.
\end{proof}

\begin{proof}[Proof of \corref{cor:gapclose}]
The family $\eps\mapsto\opH^\eps$ becomes norm-resolvent continuous after a dilation. Let $(\mathcal U_\eps u)(y)=\eps u(\eps y)$, a unitary operator on $L^2(\RR^2;\CC^2)$ that commutes with the Heaviside function $\rP$ when $x_0=0$, which we may assume. Then
\[
\tilde\opH(\eps)\coloneqq\eps^2\,\mathcal U_\eps\opH^\eps\mathcal U_\eps^{-1}=\eps D_y\cdot\sigma+(\eps^2m+\beta\Delta_y)\sigma_3+\eps\rho(\eps y)W(y),
\]
with bulks $\eps^2\,\mathcal U_\eps\opH^\infty\mathcal U_\eps^{-1}$ and $\tilde\opH_B(\eps)\coloneqq\eps^2\,\mathcal U_\eps\opH^\eps_B\mathcal U_\eps^{-1}$. These three families are norm-resolvent continuous in $\eps\in(0,\infty)$: the coefficients of the differential parts are continuous in $\eps$, the first-order terms are relatively bounded with respect to $\beta\Delta_y\sigma_3$, and $\sup_y\abs{\rho(\eps y)-\rho(\eps_\circ y)}\le\norm{\rho'}_{L^\infty}\max(\abs a,\abs b)\abs{\eps-\eps_\circ}/\min(\eps,\eps_\circ)$, since $\rho(x)=\rho(x_2)$ is constant outside $\{a\le x_2\le b\}$. Suppose that $0\notin\spec\opH^\eps_B$ for every $\eps\in[\eps',\eps'']$. Since also $0\notin\spec\opH^\infty$, continuity and compactness give $g_*>0$ such that both bulks of $\tilde\opH(\eps)$ are gapped on $(-g_*,g_*)$ for all $\eps\in[\eps',\eps'']$ \cite[Theorem~VIII.23]{RS1}. With a fixed switch function $\tilde\varphi$ whose transition lies in $(-g_*/2,g_*/2)$, the operator $T[\tilde\opH(\eps)]$ is then Fredholm \cite{bal2022topological,QB22}, and it is norm-continuous in $\eps$ by the Helffer--Sj\"ostrand formula \eqref{eq:HS}, so that its index is constant on $[\eps',\eps'']$. Since $\mathcal U_\eps$ commutes with $\rP$, ${\rm Index}\,T[\tilde\opH(\eps)]$ is the index of the operator \eqref{eq:Teps} built from $\opH^\eps$ with the switch function $\tilde\varphi(\eps^2\cdot)$. At $\eps=\eps'$ and $\eps=\eps''$, this index equals ${\rm Index}\,T^{\eps'}$, respectively ${\rm Index}\,T^{\eps''}$, since the index does not depend on the switch function as long as its transition lies in a common gap of the two bulks \cite{bal2022topological,QB22}. This contradicts ${\rm Index}\,T^{\eps'}\ne{\rm Index}\,T^{\eps''}$.

For the second statement, \cref{thm:stab} gives ${\rm Index}\,T^{\eps}=0$ for $\eps\ge\eps_1$ and ${\rm Index}\,T^{\eps}=\frac12(\sign{m_{\rm eff}}-\sign{m_+})\ne0$ for $\eps\le\eps_0$. The switch function $\varphi$ of \cref{thm:stab} has its transition in a compact subset of $(-m_0,m_0)$. The bulk $\opH^{\eps_1}_B$ has no spectrum in a neighborhood of this subset for $\eps_1$ large, since $\eps_1^{-1}W$ is then small and $\opH^\infty$ is gapped on $(-m_0,m_0)$. By \cref{thm:limit-absorpt-princ-4} with $\rho\equiv1$, $\opH^\eps_B$ converges to $\opH^0_B$ in the norm-resolvent sense, so that, by \cref{sec:topol-class-1} and \cite[Theorem~VIII.23]{RS1}, the same holds for $\opH^{\eps_0}_B$ after decreasing $\eps_0$ if necessary. The first statement then applies with $(\eps',\eps'')=(\eps_0,\eps_1)$.
\end{proof}

\section*{Acknowledgment} GB's work was supported in part by the US National Science Foundation Grants DMS-2306411 and DMS-1908736. We thank Kirill Cherednichenko for pointing out a gap in an earlier version of the proof of \cref{thm:limit-absorpt-princ-4}.

\end{document}